\documentclass[12pt,a4paper]{article}
\usepackage{amsfonts}
\usepackage{eurosym}
\usepackage{latexsym}
\usepackage{amsmath}
\usepackage{amssymb}
\usepackage[margin=.7in]{geometry}
\usepackage[english]{babel}
\usepackage{scalerel,stackengine}
\usepackage{extpfeil}
\usepackage{graphicx}
\usepackage{ntheorem}
\usepackage{tikz-cd}
\usepackage{MnSymbol}
\usepackage{stackengine}

\theoremseparator{.}

\stackMath
\newcommand\reallywidehat[1]{\savestack{\tmpbox}{\stretchto{  \scaleto{    \scalerel*[\widthof{\ensuremath{#1}}]{\kern-.6pt\bigwedge\kern-.6pt}    {\rule[-\textheight/2]{1ex}{\textheight}}  }{\textheight}}{0.5ex}}\stackon[1pt]{#1}{\tmpbox}}
\newtheorem{theorem}{Theorem}

\newtheorem{claim}[theorem]{Claim}

\newtheorem{lemma}[theorem]{Lemma}

\newtheorem{proposition}[theorem]{Proposition}
\newtheorem{remark}[theorem]{Remark}

\def\curl{\operatorname{curl}}

\def\supp{\operatorname{supp}}

\begin{document}
\makeatletter
\def\blfootnote{\xdef\@thefnmark{}\@footnotetext}
\makeatother

\date{}
\title{\textbf{Nonexistence of Henkin type projections via a Wiener theorem
for multipliers}}
\author{Eduard Curc\u a\thanks{%
Faculty of Mathematics, Informatics and Mechanics, University of Warsaw,
Poland } \and Micha\l {} Wojciechowski\thanks{%
Institute of Mathematics, Polish Academy of Sciences, Warsaw, Poland} }
\maketitle

\begin{abstract}
\setlength{\parindent}{2cm} Let $d\geq 2$, $l\geq 0$ and suppose $X$ is one
of the function spaces $W^{l,1}(\mathbb{T}^{d})$, $W^{l,\infty }(\mathbb{T}%
^{d})$ or $C^{l}(\mathbb{T}^{d})$. We extend a result of Henkin (1967),
showing that, for appropriate $N\times N$ matrix operators $A(D)$, the
subspace of $X^{N}$ consisting of $A(D)-$free elements is noncomplemented. In
order to prove this we establish a new property of the Fourier multipliers
that are bounded on $X$: the kernel $k$ of any such multiplier obeys a
weaker version of Wiener's theorem for the singularities of measures.
\end{abstract}
\makeatletter

\makeatother

\makeatletter

\makeatother

\section{Introduction}

\blfootnote{Keywords: Fourier multipliers, Sobolev spaces.}
\blfootnote{MSC 2020 classification: 42B15, 42B05, 46E35.}

In \cite{Henk}(1967) G. M. Henkin obtained the following result:

\begin{theorem}
\label{th.Henk'}Suppose $d\geq 2$ and $l\geq 0$ are integers. Then, the
space 
\begin{equation}
G_{1}(C^{l}(\mathbb{T}^{d})):=\{\nabla f\in (C^{l}(\mathbb{T}^{d}))^{d}\mid f%
\text{ distribution on }\mathbb{T}^{d}\},  \label{grad-0}
\end{equation}%
is not a complemented subspace of $(C^{l}(\mathbb{T}^{d}))^{d}$.
\end{theorem}

The proof of Henkin in \cite{Henk} (1967) is based on a clever use of the
group of the symmetries of the sphere. He considers the problem on the
sphere instead of the torus. He shows (by averaging) that, if Theorem \ref%
{th.Henk'} fails then there exists a unique corresponding projection that
commutes with the action of $O_{d}$. However, there is an explicit formula
for this projection and it turns out it can not be Lipschitz. See also \cite[%
Theorem 10.8]{B-L} for an exposition of this proof.

In the case $l=0$ there are other proofs of Theorem \ref{th.Henk'}. For
instance:

\begin{itemize}
\item[(i)] The proof of S. V. Kislyakov in \cite[Theorem 3]{Ki} (1975) via
the theory of absolutely summing operators. He showed that $C^{1}(\mathbb{T}%
^{d})$ is not an isomorphic copy of a quotient space of $C(S)$. Note that
this result is much stronger than statement of Theorem \ref{th.Henk'} when $%
l=0$.

\item[(ii)] The proof given in \cite[Section 4.2]{CE-1} (2024) via complex
interpolation. Here it is shown that the existence of a bounded projection
implies the existence of a projection that is simultaneously bounded from $%
(C(\mathbb{T}^{d}))^{d}$ to $G_{1}(C(\mathbb{T}^{d}))$ and from $(L^{2}(%
\mathbb{T}^{d}))^{d}$ to $G_{1}(L^{2}(\mathbb{T}^{d}))$. We obtain that the
couple $(C^{1}(\mathbb{T}^{d}),H^{1}(\mathbb{T}^{d}))$ interpolates
\textquotedblleft in the same way\textquotedblright\ as $(C(\mathbb{T}%
^{d}),L^{2}(\mathbb{T}^{d}))$. This is shown to be false by using standard
Littlewood-Paley and trace theory.
\end{itemize}

\bigskip

In this paper we study results of the same nature as Theorem \ref{th.Henk'}.
In particular, we give a much simpler proof of Theorem \ref{th.Henk'} than
the proofs listed above (see Remark \ref{rem.X}). Instead of the symmetries
of the sphere, we use (in a rather standard way) the symmetries of the
torus. By this the problem gets reduced to a problem concerning the Fourier
multipliers on $C^{l}(\mathbb{T}^{d})$. Then, inspired by an idea that goes
back to J.-P. Kahane (\cite[p. 401 and Theorem 3.1]{P}; see also \cite[%
Section 3]{KW-1} for a similar application), we conclude by using a variant
of Wiener's theorem on the singularities of measures.

Similar arguments give us more. Our main result is a generalisation of
Theorem \ref{th.Henk'} that does not seem to be reachable by the methods
(i), (ii) listed above neither by the method of Henkin. In order to
formulate our result let us introduce some notation.

Let $X$ be a Banach function space on $\mathbb{T}^{d}$. For a matrix
function $A:\mathbb{R}^{d}\rightarrow M_{N}(\mathbb{C)}$, we define the
space of all $A(D)$-free vectors as 
\begin{equation*}
W_{A}(X):=\left\{ u\in X^{N}\text{ }|\text{ }A(D)u^{\dag }=0\right\}
\subseteq X^{N},
\end{equation*}%
where the norm is induced by the norm of $X$. (Here and in what follows, by
\textquotedblleft\ $\dag $ \textquotedblright\ we mean the usual
transposition of vectors.) Clearly, $W_{A}(X)$ is a closed subspace of $%
X^{N} $. With this notation our main result is the following.

\begin{theorem}
\label{th.A} Let $d,N\geq 2$, $l\geq 0$ be integers and consider a matrix
function $A:\mathbb{R}^{d}\rightarrow M_{N}(\mathbb{C)}$ such that the
following conditions are satisfied:

\begin{itemize}
\item[(A1)] $A$ is a $0-$homogeneous function continuous on $\mathbb{S}%
^{d-1} $;

\item[(A2)] $\bigcap_{\omega \in \mathbb{S}^{d-1}}KerA(\omega )=\{0\}$;

\item[(A3)] $A$ is non-invertible on some nonempty open subset of $\mathbb{S}%
^{d-1}$.
\end{itemize}

Suppose $X$ is one of the spaces $W^{l,1}(\mathbb{T}^{d})$, $W^{l,\infty }(%
\mathbb{T}^{d})$ or $C^{l}(\mathbb{T}^{d})$. Then, $W_{A}(X)$ is not
complemented in $X^{N}$.
\end{theorem}

\begin{remark}
\label{rem.A}The conditions (A1)--(A3) above are not intended to be sharp.
We chose sufficiently general conditions that are also easy to work with.
Considering the operator $A(D)$ obtained by completing the operator $|\nabla
|^{-1}\curl $ with zeros to a square matrix, we obtain Theorem \ref%
{th.Henk'} for $X=C^{l}(\mathbb{T}^{d})$. We note that Theorem \ref{th.A}
also covers the case of $W^{l,1}(\mathbb{T}^{d})$. A closer look at Henkin's
proof shows that his argument is still valid if we replace $C^{l}(\mathbb{T}%
^{d})$ by $W^{l,1}(\mathbb{T}^{d})$ in Theorem \ref{th.Henk'} (see the comments after the proof of Theorem \ref{th.GC}). However, Henkin's proof seems to be confined to the
gradient spaces.
\end{remark}

The main ingredient in the proof of Theorem \ref{th.A} is the fact that,
even if the Fourier multipliers on the spaces $W^{l,1}(\mathbb{T}^{d})$, $%
W^{l,\infty}(\mathbb{T}^{d})$ or $C^{l}(\mathbb{T}^{d})$ (with $l\geq 1$
integer) are not given by convolution with kernels that are measures, the
Fourier transform of such kernels still satisfy a weaker form of the Wiener
theorem (that will be called \textit{Wiener property}; see Lemmas \ref%
{lem.W'}, \ref{lem.W'-1} in Section \ref{sec.W}). Namely, if $T_{k}$ (the
operator given by convolution with the kernel $k$) is a bounded multiplier
on $W^{l,1}(\mathbb{T}^{d})$, $W^{l,\infty}(\mathbb{T}^{d})$ or $C^{l}(%
\mathbb{T}^{d})$, then the limit of the averages on balls 
\begin{equation}
\lim_{t\rightarrow \infty }\dfrac{1}{|B_{d}(t\omega ,r(t))|}\sum_{\chi\in
B_{d}(t\omega ,r(t))}\widehat{k}(\chi ) ,  \label{Wi}
\end{equation}%
exists and does not depend on $\omega \in \mathbb{S}^{d-1}$, where $r$ is
some appropriate function with $\lim_{t\rightarrow \infty }r(t)=\infty $
that depends only on the function space considered.

This fact, that is interesting for itself, turns out to provide an almost
immediate proof of Theorem \ref{th.A}. It is the Wiener property of the Fourier multipliers on the limiting Sobolev spaces that is the most important result in this paper.

In short: Section \ref{sec.fm} collects general facts and formulas concerning the multiplier that are used later. In Section \ref{sec.W} we introduce the Wiener property and we prove various forms of it; first in the case of $W^{l,\infty}$ (and $C^{l}$), then in the case of $W^{l,1}$. In Section \ref{sec.proj} we use the Wiener property in order to prove
Theorem \ref{th.A} (and Theorem \ref{th.GC}).

 In all the proofs in this paper we
use only standard measure and distributions theory.

\bigskip

\noindent \textbf{Notation and conventions.} Throughout the paper we use
mainly standard notation. For instance, we often use the symbols $\lesssim $
and $\sim $. For two non-negative variable quantities $a$ and $b$ we write $%
a\lesssim b$ if there exists a constant $C>0$ such that $a\leq Cb$. If $%
a\lesssim b$ and $b\lesssim a$, then we write $a\sim b$.

We will often use the standard basis of $\mathbb{R}^{d}$, i.e., $\mathbf{e}%
_{j}:=(0,...,0,1,0,...,0)$ (with $1$ on the $j-$th position), for $j=1,...,d$%
. For any $\xi \in \mathbb{R}^{d}$ we define $\langle
\xi\rangle:=(1+|\xi|^{2})^{1/2}$.

The ball of center\ $\xi \in \mathbb{R}^{d}$ and radius $r>0$ will be
denoted by $B(\xi ,r)$ and its discrete version will be denoted by $%
B_{d}(\xi ,r):=B(\xi ,r)\cap \mathbb{Z}^{d}$. Also, $|B(\xi ,r)|$ is the
volume of $B(\xi ,r)$ and $|B_{d}(\xi ,r)|$ is the number of points in $%
B_{d}(\xi ,r)$.

We use the standard symbol for the integral mean on a bounded domain $\Omega
\subset \mathbb{R}^{d}$: 
\begin{equation*}
\strokedint_{\Omega }f(x)dx=\frac{1}{|\Omega |}\int_{\Omega }f(x)dx,
\end{equation*}%
when $f\in L_{loc}^{1}(\mathbb{R}^{d})$. The Fourier transform that we work
with is given by the following formula 
\begin{equation*}
\widehat{f}(\xi ):=\int_{\mathbb{R}^{d}}e^{-i\langle x,\xi \rangle }f(x)dx,
\end{equation*}%
for any sufficiently regular function $f$. We consider a similar definition
on $\mathbb{T}^{d}$. We make the convention that, whenever we consider the
operators $T_{k}$ (given by convolution with the kernel $k$) the kernel $k$
is assumed to be a tempered distribution (and $\widehat{k}\in
L_{loc}^{1}\left( \mathbb{R}^{d}\right) $) in the case $\mathbb{R}^{d}$ and
a distribution in the case of $\mathbb{T}^{d}$.

Throughout the paper by $c$ we denote various constants. The dimension $d$
will be always considered to be at least $2$ and $l$ will be a nonnegative
integer. Other notation will be introduced when needed.

\section*{Acknowledgements}

This work was supported by the National Science Centre, Poland, CEUS
programme, project no. 2020/02/Y ST1/00072.

\section{Function spaces and Fourier multipliers}
\label{sec.fm}

\smallskip Let us briefly recall the definitions of the function spaces we
use and several elementary facts concerning the Fourier multipliers. Even if
our main result Theorem \ref{th.GC} refers to the torus $\mathbb{T}^{d}$,
for some technical reasons we need to consider also functions spaces that
are defined on $\mathbb{R}^{d}$.

Fix some parameter $p\in \lbrack 1,\infty ]$. Given a nonnegative integer $l$%
, the inhomogeneous space $W^{l,p}(\mathbb{R}^{d})$ consists of those
tempered distributions $f$ on $\mathbb{R}^{d}$ for which $\nabla ^{\alpha
}f\in L^{p}(\mathbb{R}^{d})$, for any multiindex $\alpha $, with $\left\vert
\alpha \right\vert \leq l$. This space is endowed with the norm given by%
\begin{equation*}
\left\Vert f\right\Vert _{W^{l,p}}=\max_{\alpha \in \mathbb{N}%
^{d},\left\vert \alpha \right\vert \leq l}\left\Vert \nabla ^{\alpha
}f\right\Vert _{L^{p}}.
\end{equation*}

Similarly, the homogeneous space $\dot{W}^{l,p}(\mathbb{R}^{d})$ consists of
those tempered distributions $f$ on $\mathbb{R}^{d}$ for which $\nabla
^{\alpha }f\in L^{p}(\mathbb{R}^{d})$, for any multiindex $\alpha $, with $%
\left\vert \alpha \right\vert =l$. This space is endowed with the seminorm
given by%
\begin{equation*}
\left\Vert f\right\Vert _{\dot{W}^{l,p}}=\left\Vert \nabla ^{l}f\right\Vert
_{L^{p}}=\max_{\alpha \in \mathbb{N}^{d},\left\vert \alpha \right\vert
=l}\left\Vert \nabla ^{\alpha }f\right\Vert _{L^{p}}.
\end{equation*}

On $\mathbb{R}^{d}$ we also consider the homogeneous space $\dot{C}_{0}^{l}(%
\mathbb{R}^{d})$ and the inhomogeneous space $C_{0}^{l}(\mathbb{R}^{d})$.
Here, $\dot{C}_{0}^{l}(\mathbb{R}^{d})$ is defined as the closure of the
space of Schwartz functions in the norm induced by $\dot{W}^{l,\infty }(%
\mathbb{R}^{d})$. The space $C_{0}^{l}(\mathbb{R}^{d})$ is the space of
those functions $f$ on $\mathbb{R}^{d}$ for which $\nabla ^{\alpha }f\in
C_{0}(\mathbb{R}^{d})$ (the space of the continuous functions that converge
to $0$ at infinity), for any multiindex $\alpha $, with $\left\vert \alpha
\right\vert \leq l$. The norm of $C_{0}^{l}(\mathbb{R}^{d})$ is induced by
the norm of the inhomogeneous space $W^{l,\infty }(\mathbb{R}^{d})$. On $%
\mathbb{T}^{d}$ we will consider the spaces $W^{l,p}(\mathbb{T}^{d})$, $%
C^{l}(\mathbb{T}^{d})$ with their standard definition.

Since $(C_{0}(\mathbb{R}^{d}))^{\ast }=\mathcal{M}(\mathbb{R}^{d})$, where $%
\mathcal{M}(\mathbb{R}^{d})$ is the space of the Radon measures on $\mathbb{R%
}^{d}$, we get that $(C_{0}^{l}(\mathbb{R}^{d}))^{\ast }=\mathcal{M}^{-l}(%
\mathbb{R}^{d})$, where $\mathcal{M}^{-l}(\mathbb{R}^{d})$ is the space of
all distributions $f$ of the form 
\begin{equation}
f=\sum_{|\alpha |\leq l}\nabla ^{\alpha }\mu _{\alpha },  \label{defM-1}
\end{equation}%
where each $\mu _{\alpha }$ belongs to $\mathcal{M}(\mathbb{R}^{d})$ (see
for instance \cite[Section 4.3]{Z}). The norm on $\mathcal{M}^{-l}(\mathbb{R}%
^{d})$ is given by 
\begin{equation}
\left\Vert f\right\Vert _{\mathcal{M}^{-l}}:=\inf \left\{ \left.
\sum_{|\alpha |\leq l}\left\Vert \mu _{\alpha }\right\Vert _{\mathcal{M}}%
\text{ }\right\vert \text{ }f=\sum_{|\alpha |\leq l}\nabla ^{\alpha }\mu
_{\alpha }\right\} \text{.}  \label{defM-2}
\end{equation}

Similarly, on $\mathbb{T}^{d}$ we have $(C^{l}(\mathbb{T}^{d}))^{\ast }=%
\mathcal{M}^{-l}(\mathbb{T}^{d})$, for any integer $l\geq 0$, where $%
\mathcal{M}(\mathbb{T}^{d})$ is the space of the Radon measures on $\mathbb{T%
}^{d}$ and $\mathcal{M}^{-l}(\mathbb{T}^{d})$ is defined as in (\ref{defM-1}%
) and (\ref{defM-2}), with the obvious adaptations to the case of $\mathbb{T}%
^{d}$.

\bigskip

Let $X$ and $Y$ be a seminormed spaces of tempered distributions on $\mathbb{%
R}^{d}$, or distributions on $\mathbb{T}^{d}$. If $k$ is a (tempered)
distribution and $\sigma :=\widehat{k}$ (with $\sigma \in L_{loc}^{1}(%
\mathbb{R}^{d})$, in the case $k$ is defined on $\mathbb{R}^{d}$), we
define the operator $T_{k}$ by\footnote{%
In the case of $\mathbb{R}^{d}$ we impose that $\sigma \widehat{f}$ is a
tempered distribution when $f$ is Schwartz.} 
\begin{equation*}
\widehat{T_{k}f}=\sigma \widehat{f},\ 
\end{equation*}%
for any test function $f$ (that is $f$ is Schwartz in the case of $\mathbb{R}%
^{d}$ and $f\in X$ in the case of $\mathbb{T}^{d}$). If for any test
function $f$ we have 
\begin{equation}
\Vert T_{k}f\Vert _{Y}\leq C\Vert f\Vert _{X},\   \label{Fmult}
\end{equation}%
for some constant $C<\infty $, we say that $T_{k}$ is a (bounded) Fourier
multiplier from $X$ to $Y$, and we write $T_{k}\in M(X,Y)$. The best
constant $C$ in (\ref{Fmult}) is the norm of $T_{k}$ in the space $M(X,Y)$,
denoted by $\Vert T_{k}\Vert _{M(X,Y)}$. If $Y=X$, then we write $M(X)$
instead of $M(X,X)$.

The distribution $\sigma $ will be called the symbol of the multiplier $%
T_{k} $. Sometimes it is convenient to write $\sigma (D)$, with $D:=\nabla
/i $, in place of $T_{k}$.

\bigskip

By using the Ornstein $L^{1}-$noninequality, Bonami and Poornima (\cite[%
Theorem 2]{BP}) proved that if a homogeneous symbol induces a bounded
multiplier on $\dot{W}^{l,1}(\mathbb{R}^{d})$, then the symbol must be a
constant function. We record this result with \cite[Lemma 6]{CE-F} as
follows.

\begin{lemma}
\label{lem.Bon-Poo}Let $k$ be a scalar tempered distribution on $\mathbb{R}%
^{d}$ such that $T_{k}\in M(\dot{W}^{l,1}(\mathbb{R}^{d}))$. Then, we have $%
\widehat{k}\in C_{b}(\mathbb{R}^{d}\backslash \{0\})$. If moreover, $%
\widehat{k}$ is homogeneous of degree $0$, then $\widehat{k}$ is constant.
\end{lemma}

\bigskip

One also need a similar result concerning the seminorm of $\dot{W}^{l,\infty
}(\mathbb{R}^{d})$:

\begin{lemma}
\label{lem.jfa}Let $k$ be a scalar tempered distribution on $\mathbb{R}^{d}$
such that $T_{k}$ is bounded from $\dot{C}_{0}^{l}(\mathbb{R}^{d})$ to $\dot{%
W}^{l,\infty }(\mathbb{R}^{d})$. Then, we have $\widehat{k}\in C_{b}(\mathbb{%
R}^{d}\backslash \{0\})$. If moreover, $\widehat{k}$ is homogeneous of
degree $0$, then $\widehat{k}$ is constant.
\end{lemma}

The first statement of Lemma \ref{lem.jfa} is the analogue of \cite[Lemma 8]%
{CE-F} and the second statement is the analogue of \cite[Theorem 10]{CE-F}.
Here, $\dot{C}_{0}^{l}(\mathbb{R}^{d})$ is replacing the domain ($\dot{W}%
^{l,\infty }(\mathbb{R}^{d})$ in \cite{CE-F}) of the operator $T_{k}$.
However, since in the proofs of \cite[Lemma 8]{CE-F}, \cite[Theorem 10]{CE-F}
we only use the action of the operator on the Schwartz functions, we get as
well Lemma \ref{lem.jfa} above by the same arguments as in \cite{CE-F}. It
is worth recalling, that the proofs in \cite{CE-F} related to Lemma \ref%
{lem.jfa} do not use the Ornstein noninequality, but much simpler techniques
related to $L^{\infty }-$noninequalities as in \cite{dL-M}\footnote{%
In fact, similar methods can be seen in the proofs of Propositions \ref%
{prop.g}, \ref{prop.b-inhom} below.}

An immediate consequence of the first part of Lemma \ref{lem.jfa} is that
the boundedness of $T_{k}$ from $\dot{C}_{0}^{l}(\mathbb{R}^{d})$ to $\dot{W}%
^{l,\infty }(\mathbb{R}^{d})$ gives $T_{k}\in \dot{C}_{0}^{l}(\mathbb{R}%
^{d}) $. Indeed, since $\widehat{k}\in C_{b}(\mathbb{R}^{d}\backslash \{0\})$%
, for any Schwartz function $\psi $ on $\mathbb{R}^{d}$, the function

\begin{equation*}
T_{k}\psi (x)=\int_{\mathbb{R}^{d}}\widehat{k}(\xi )\widehat{\psi }(\xi
)e^{i\left\langle x,\xi \right\rangle }d\xi =\widehat{k}(D)\psi (x),
\end{equation*}%
is Schwartz on $\mathbb{R}^{d}$. We get that $T_{k}$ is bounded from $\dot{C}%
_{0}^{l}(\mathbb{R}^{d})$ to $\dot{C}_{0}^{l}(\mathbb{R}^{d})$.

In order to gather general facts about the multipliers on $C_{0}^{l}(\mathbb{%
R}^{d})$ we formulate the following\footnote{%
Similar facts were already considered in \cite{Poo}.}:

\begin{proposition}
\label{prop.g}Let $k$ be a scalar tempered distribution on $\mathbb{R}^{d}$.
Then, the following are equivalent:

\begin{itemize}
\item[(i)] $T_{k}$ is bounded from $C_{0}^{l}(\mathbb{R}^{d})$ to $C_{0}^{l}(%
\mathbb{R}^{d})$ (i.e., $T_{k}\in M(C_{0}^{l}(\mathbb{R}^{d}))$);

\item[(ii)] For any multiindex $\alpha $ with $|\alpha |\leq l$ we have $%
\nabla ^{\alpha }k\in \mathcal{M}^{-l}(\mathbb{R}^{d})$.
\end{itemize}

(A similar equivalence holds on the torus where $C_{0}^{l}(\mathbb{R}^{d})$, 
$\mathcal{M}^{-l}(\mathbb{R}^{d})$ are replaced by $C^{l}(\mathbb{T}^{d})$, $%
\mathcal{M}^{-l}(\mathbb{T}^{d})$ respectively.)
\end{proposition}

\bigskip

\noindent\textbf{Proof.} We write 
\begin{equation*}
|\left\langle \nabla ^{\alpha }k,f\right\rangle |=|\nabla ^{\alpha }k\ast 
\overline{f}(0)|\leq \left\Vert \nabla ^{\alpha }(k\ast \overline{f}%
)\right\Vert _{L^{\infty }}\lesssim \left\Vert f\right\Vert _{W^{l,\infty }},
\end{equation*}%
for any $\alpha $ with $|\alpha |\leq l$ and any $f\in \dot{C}_{0}^{l}(%
\mathbb{R}^{d})$. Hence, $\nabla ^{\alpha }k\in (C_{0}^{l}(\mathbb{R}%
^{d}))^{\ast }=\mathcal{M}^{-l}$. This shows that (i) implies (ii).

Now, fix some $\alpha $ as before and write 
\begin{equation}
\nabla ^{\alpha }k=\sum_{|\beta |\leq l}\nabla ^{\beta }\nu _{\beta ,\alpha
},  \label{dec-miu}
\end{equation}%
for some $\nu _{\beta ,\alpha }\in \mathcal{M}(\mathbb{R}^{d})$ (depending
on $k$ and $\alpha $). We have 
\begin{equation*}
\left\Vert \nabla ^{\alpha }(k\ast f)\right\Vert _{L^{\infty }}\leq
\sum_{|\beta |\leq l}\left\Vert \nu _{\beta ,\alpha }\ast \nabla ^{\beta
}f\right\Vert _{L^{\infty }}\lesssim \sum_{|\beta |\leq l}\left\Vert \nu
_{\beta ,\alpha }\right\Vert _{\mathcal{M}}\left\Vert \nabla ^{\beta
}f\right\Vert _{L^{\infty }}\lesssim \left\Vert f\right\Vert _{W^{l,\infty
}},
\end{equation*}%
for any $f\in C_{0}^{l}(\mathbb{R}^{d})$. This shows that (ii) implies (i).
(In the case of the torus the proof is similar.)\hfill $\square $

\begin{remark}
\bigskip The distributions $k$ satisfying condition (ii) in Proposition \ref%
{prop.g} were also considered in \cite[p. 15]{Poo} where the space of such
elements was denoted by $B^{l}$. It was also shown in \cite[Proposition 5.6]%
{Poo} that $B^{1}$ describes the space of multipliers of certain Segal
algebra. Proposition \ref{prop.g} above shows that the elements of $B^{l}$
in \cite{Poo} are the multipliers on $C_{0}^{l}(\mathbb{R}^{d})$ (or $C^{l}(%
\mathbb{T}^{d})$).
\end{remark}

\begin{remark}
\label{rem.g}The proof of the the fact that \textquotedblleft (ii) implies
(i)\textquotedblright\ in Proposition \ref{prop.g} shows that we have in
fact the embedding 
\begin{equation*}
M(C^{l}(\mathbb{T}^{d}))\subseteq M(W^{l,p}(\mathbb{T}^{d})),
\end{equation*}%
for any $p\in \lbrack 1,\infty ]$. See also Proposition 3.2 in \cite{Poo}
for similar considerations.
\end{remark}

\begin{remark}
\label{rem.g.inf}The proof of Proposition \ref{prop.g} shows in fact that we
have 
\begin{equation*}
M(C_{0}^{l}(\mathbb{R}^{d}))=M(C_{0}^{l}(\mathbb{R}^{d}),W^{l,\infty }(%
\mathbb{R}^{d}))=M(W^{l,\infty }(\mathbb{R}^{d})),
\end{equation*}%
and 
\begin{equation*}
M(C^{l}(\mathbb{T}^{d}))=M(C^{l}(\mathbb{T}^{d}),W^{l,\infty }(\mathbb{T}%
^{d}))=M(W^{l,\infty }(\mathbb{T}^{d})).
\end{equation*}
\end{remark}

The fact that the kernels of the bounded multipliers on spaces like $%
C_{0}^{l}(\mathbb{R}^{d})$ and $W^{l,1}(\mathbb{R}^{d})$ are pseudomeasures
is somewhat folklore. For the convenience of the reader we provide below a
proof of this fact. We will also use latter in the paper several formulas
that appear in this proof.

\begin{proposition}
\label{prop.b-inhom}If $k$ is a tempered distribution on $\mathbb{R}^{d}$
(or distribution on $\mathbb{T}^{d}$) such that $T_{k}\in M(C_{0}^{l}(%
\mathbb{R}^{d}))$ or $T_{k}\in M(W^{l,1}(\mathbb{R}^{d}))$ ($T_{k}\in
M(C^{l}(\mathbb{T}^{d}))$ or $T_{k}\in M(W^{l,1}(\mathbb{T}^{d}))$), then $%
\widehat{k}\in C_{b}(\mathbb{R}^{d})$ ($\widehat{k}\in \ell^{\infty }(%
\mathbb{Z}^{d})$ in the case of the torus).
\end{proposition}

\noindent \textbf{Proof.} This follows by the same arguments as the ones in
the proof of \cite[Lemma 8]{CE-F}. In fact, the argument can be easily
extracted from the equality (\ref{dec-miu}) above. Indeed, writing (\ref%
{dec-miu}) for $\alpha =0$ and taking the Fourier transform, we get%
\begin{equation*}
\widehat{k}(\eta )=\sum_{|\beta |\leq l}(-i\eta )^{\beta }\widehat{\nu }%
_{\beta }(\eta ),
\end{equation*}%
on $\mathbb{R}^{d}$. This shows that $\widehat{k}\in C_{b}(B(0,2))$. By a
similar argument, we can obtain the continuity and boundedness of $\widehat{k%
}$ far away from the origin. Indeed, using (\ref{dec-miu}) for $\alpha =l%
\mathbf{e}_{j}$ we have 
\begin{equation}
\partial _{j}^{l}k=\sum_{|\beta |\leq l}\nabla ^{\beta }\nu _{\beta ,j},
\label{k-j-dec}
\end{equation}%
for some measures $\nu _{\beta ,j}\in \mathcal{M}(\mathbb{R}^{d})$
(depending on $k$ and $j$). Taking the Fourier transform we get%
\begin{equation*}
(-i\eta _{j})^{l}\widehat{k}(\eta )=\sum_{|\beta |=l}(-i\eta )^{\beta }%
\widehat{\nu }_{\beta ,j}(\eta )+\sum_{|\beta |<l}(-i\eta )^{\beta }\widehat{%
\nu }_{\beta ,j}(\eta ),
\end{equation*}%
on $\mathbb{R}^{d}$. In particular, for any $\eta \in \mathbb{R}^{d}$ with $%
\eta _{j}\neq 0$ we have 
\begin{equation}
\widehat{k}(\eta )=\sum_{|\beta |=l}\frac{\eta ^{\beta }}{\eta _{j}^{l}}%
\widehat{\nu }_{\beta ,j}(\eta )+\sum_{|\beta |<l}\frac{(-i\eta )^{\beta }}{%
(-i\eta _{j})^{l}}\widehat{\nu }_{\beta ,j}(\eta ).  \label{k-niu}
\end{equation}

It suffices to verify the continuity and boundedness of $\widehat{k}$ on the
sets $A_{j}:=\{\eta \mid |\eta _{j}|>|\eta |/4\}\backslash B(0,1)$, for each 
$j\in \{1,...,d\}$.

Now, suppose that $T_{k}\in M(W^{l,1}(\mathbb{R}^{d}))$. We first note that
for any $f\in C_{c}^{\infty }(\mathbb{R}^{d})$ and any $t>0$ we have 
\begin{equation*}
\widehat{k}(t\cdot D)f(x)=\widehat{k}(D)f(t\cdot )(x/t),
\end{equation*}%
on $\mathbb{R}^{d}$. Hence, for $t\geq 1$ we can write 
\begin{eqnarray*}
\left\Vert \widehat{k}(t\cdot D)f\right\Vert _{\dot{W}^{l,1}} &\leq
&\left\Vert \widehat{k}(D)f(t\cdot )(\cdot /t)\right\Vert
_{W^{l,1}}=t^{d-l}\left\Vert \widehat{k}(D)f(t\cdot )\right\Vert _{W^{l,1}},
\\
&\lesssim &t^{d-l}\left\Vert f(t\cdot )\right\Vert _{M(W^{l,1})}\sim
\sum_{j=0}^{d}t^{j-l}\left\Vert \nabla ^{j}f\right\Vert _{L^{1}}\lesssim
\left\Vert f\right\Vert _{W^{l,1}},
\end{eqnarray*}
and therefore 
\begin{equation}
\left\Vert \widehat{k}(t\cdot D)\right\Vert _{W^{l,1}\rightarrow \dot{W}%
^{l,1}}\leq C,  \label{1-bound}
\end{equation}
for some real constant $C>0$, uniformly for any $t\geq 1$.

Now, suppose $m$ is a symbol such that $m(D):W^{l,1}(\mathbb{R}%
^{d})\rightarrow \dot{W}^{l,1}(\mathbb{R}^{d})$. If $\rho $ is a Schwartz
function with $\widehat{\rho }>0$ on $\mathbb{R}^{d}$, then we have $\mu
_{j}:=\partial _{j}^{l}m(D)\rho \in \mathcal{M}(\mathbb{R}^{d})$, for $j\in
\{1,..,d\}$, and (as in \cite[Lemma 6]{CE-F}) 
\begin{equation*}
\left\Vert m\right\Vert _{L^{\infty }(2B\backslash \{|\xi
_{j}|<1/2\})}=\left\Vert \frac{\widehat{\mu }_{j}}{\xi _{j}{}^{l}\widehat{%
\rho }}\right\Vert _{L^{\infty }(2B\backslash \{|\xi _{j}|<1/2\})}\lesssim
\left\Vert \widehat{\mu }_{j}\right\Vert _{L^{\infty }(2B\backslash \{|\xi
_{j}|<1/2\})}\lesssim \left\Vert \mu _{j}\right\Vert _{\mathcal{M}}\leq C.
\end{equation*}

By adding up over all $j\in \{1,..,d\}$, we get $\left\Vert m\right\Vert
_{L^{\infty }(2B\backslash B)}\lesssim C$. Applying this to $m(D)=\widehat{k}%
(t\cdot D)$, for any $t\geq 1$, we obtain by (\ref{1-bound}) that 
\begin{equation*}
\left\Vert \widehat{k}\right\Vert _{L^{\infty }(\mathbb{R}^{d}\backslash
B)}\lesssim C.
\end{equation*}

It remains to notice that, since $\widehat{k}(D)\rho \in \mathcal{M}(\mathbb{%
R}^{d})$, by taking the Fourier transform, we obtain the continuity on $%
\mathbb{R}^{d}$ and the boundedness on $2B$ of $\widehat{k}$.

In the case of the torus the argument is straightforward. For $X=W^{l,1}(%
\mathbb{R}^{d})$, $W^{l,\infty }(\mathbb{R}^{d})$ or $C^{l}(\mathbb{T}^{d})$%
, we have $0<\left\Vert e_{\chi }\right\Vert _{X}<\infty $, for any $\chi
\in \mathbb{Z}^{d}$, where $e_{\chi }=\exp (i\left\langle \chi ,\cdot
\right\rangle )$. Therefore, if $T_{k}\in M(X)$ we have%
\begin{equation*}
|\widehat{k}(\chi )|\left\Vert e_{\chi }\right\Vert _{X}=\left\Vert
T_{k}e_{\chi }\right\Vert _{X}\leq \left\Vert T_{k}\right\Vert
_{M(X)}\left\Vert e_{\chi }\right\Vert _{X},
\end{equation*}
which gives us $|\widehat{k}(\chi )|\leq \left\Vert T_{k}\right\Vert _{M(X)}$%
, for any $\chi \in \mathbb{Z}^{d}$. \hfill $\square $

\section{The Wiener property}

\label{sec.W}

The Wiener lemma on the singularities of measures (see Lemma \ref{lem.Wien})
implies that, for any given finite measure $\mu $ on $\mathbb{R}^{d}$ and
for any $\omega \in \mathbb{S}^{d-1}$, the limit 
\begin{equation}
\lim_{t,r\rightarrow \infty }\strokedint_{B(t\omega ,r)}\widehat{\mu }(\eta
)d\eta ,  \label{t-r}
\end{equation}%
exists and does not depend on $\omega $.

One can ask whether there are other types of distributions that have a
similar property. In (\ref{t-r}) the parameters $t$ and $r$ need not be
related. In what follows we will consider somewhat weaker forms of (\ref{t-r}%
) when the radius $r$ is a function on $t$ with certain reasonable
properties. We say that a function $r:(0,\infty )\rightarrow (0,\infty )$ is
a \textit{growth function} if $r$ is nondecreasing and $\lim_{t\rightarrow
\infty }r(t)=\infty $. Using this type of functions to control the growth of
the balls we introduce the \textit{Wiener property}:

\noindent \textbf{The Wiener property on }$\mathbb{R}^{d}$\textbf{. }Let $k$
be a tempered distribution on $\mathbb{R}^{d}$ such that $\widehat{k}\in
L_{loc}^{1}(\mathbb{R}^{d})$. We say that $k$ has the Wiener property with
respect to the growth function $r:(0,\infty )\rightarrow (0,\infty )$, if
for any fixed $\omega \in \mathbb{S}^{d-1}$ the limit 
\begin{equation}
\lim_{t\rightarrow \infty }\strokedint_{B(t\omega ,r(t))}\widehat{k}(\eta
)d\eta ,  \label{WB-1}
\end{equation}%
exists and does not depend on $\omega $.

\bigskip

\noindent \textbf{The Wiener property on }$\mathbb{T}^{d}$\textbf{. }Let $k$
be a distribution on $\mathbb{T}^{d}$. We say that $k$ has the Wiener
property with respect to the growth function $r:(0,\infty )\rightarrow
(0,\infty )$, if for any fixed $\omega \in \mathbb{S}^{d-1}$ the limit%
\footnote{%
Recall that $B_{d}:=B\cap \mathbb{Z}^{d}$.} 
\begin{equation}
\lim_{t\rightarrow \infty }\frac{1}{|B(t\omega ,r(t))|}\sum_{\chi \in
B_{d}(t\omega ,r(t))}\widehat{k}(\chi ),  \label{WB-2}
\end{equation}%
exists and does not depend on $\omega $.

\bigskip

In the case a tempered distribution $k$ has the Wiener on $\mathbb{R}^{d}$
(or $\mathbb{T}^{d}$) with respect to a growth function $r$, the common
value of the limits (\ref{WB-1}) (or (\ref{WB-2})) will be called the 
\textit{Wiener number} of $k$ with respect to $r$. By extrapolation we will
say that a multiplier $T_{k}=\widehat{k}(D)$ has the Wiener property if its
corresponding convolution kernel $k$ has the Wiener property. The same
applies to its Wiener number.

The bounded multipliers of $L^{1}(\mathbb{T}^{d})$ or $L^{\infty }(\mathbb{T}%
^{d})$ are given by convolution with finite measures. Hence, by Wiener's
lemma any bounded multiplier on $L^{p}(\mathbb{T}^{d})$, with $p\in
\{1,\infty \}$, has the Wiener property with respect to any growth function.
On the other hand, if $p\in (1,\infty )$, the elements of $M(W^{l,p}(\mathbb{%
T}^{d}))=M(L^{p}(\mathbb{T}^{d}))$ might fail having a similar property. For
instance, the Riesz projection $P_{+}=\mathbf{1}_{[0,\infty )^{d}}(D)$
belongs to $M(L^{p}(\mathbb{T}^{d}))$ while the averages of $\mathbf{1}%
_{[0,\infty )^{d}}$ on the balls $B_{d}(t\omega ,r(t)) $ have limits that
depend on $\omega \in \mathbb{S}^{d-1}$ as long as $r$ does not grow too
fast. If $l\in \mathbb{N}^{\ast }$, the elements of $M(W^{l,p}(\mathbb{T}%
^{d}))$, with $p\in \{1,\infty \}$, are not in general given by convolution
with measures. Despite this we will show that they still have Wiener
property with respect to various growth functions. In our applications it is
important to not have growth functions that grow too fast. In general we
would like to have growth functions $r$ of sub-linear growth, i.e., 
\begin{equation*}
\lim_{t\rightarrow \infty }\dfrac{r(t)}{t}=0.
\end{equation*}

Nevertheless, if this is not achieved, a \textquotedblleft
weaker\textquotedblright\ condition still suffices in practice. Namely one
can consider instead the Wiener property with respect to each element of a
family $(r_{\varepsilon })_{\varepsilon \in (0,\varepsilon_0)}$ (for some $%
\varepsilon_0>0$) of growth functions with 
\begin{equation}
\limsup_{t\rightarrow \infty }\dfrac{r_{\varepsilon }(t)}{t}\leq \varepsilon
.  \label{r-epsilon}
\end{equation}

\begin{remark}
\label{rem.r-ep}The linear growth functions $r_{\varepsilon }(t)=\varepsilon
t$, for $\varepsilon \in (0,\varepsilon_0)$, form a family satisfied (\ref%
{r-epsilon}). Note also that, by defining $r_{\varepsilon }(t)=r(t)$, for
any $\varepsilon \in (0,1)$, where $r$ is a sublinear growth function, then (%
\ref{r-epsilon}) is trivially satisfied.
\end{remark}


\subsection{The Wiener property of multipliers on $W^{l,\infty }(\mathbb{T}%
^{d})$}

\label{sec.inf} In the proof of Theorem \ref{th.GC} we only use the Wiener
property on $\mathbb{T}^{d}$. However, for technical reasons, it is
convenient to deal first with the Euclidean version of this property. Also,
in this case it is convenient to work with the space $C_{0}^{l}(\mathbb{R}%
^{d})$ rather than $W^{l,\infty }(\mathbb{R}^{d})$.

\begin{lemma}
\label{lem.W} If $T_{k}\in M(C_{0}^{l}(\mathbb{R}^{d}))$, then $k$ has the
Wiener property on $\mathbb{R}^{d}$ with respect to any growth function $r$
for which $\lim_{t\rightarrow \infty }(r(t)/t)=0$.
\end{lemma}

Note that, by Proposition \ref{prop.g}, $\widehat{k}\in L_{loc}^{1}(\mathbb{R%
}^{d})$ and hence, the means of $\widehat{k}$ are well-defined on any ball.
We first need to prove that the limit of the averages of $\widehat{k}$ exist
in any direction. This is achieved by the following lemma.

\begin{lemma}
\label{lem.mean} Let $r$ be a growth function such that $\lim_{t\rightarrow
\infty }(r(t)/t)=0$. If $T_{k}\in M(C_{0}^{l}(\mathbb{R}^{d}))$, then for
any fixed $\xi \in \mathbb{R}^{d}\backslash \{0\}$ the limit 
\begin{equation}
m^{\infty }(\xi ):=\lim_{t\rightarrow \infty }\strokedint_{B(t\xi ,r(t))}%
\widehat{k}(\eta )d\eta ,  \label{Wiener-m}
\end{equation}%
exists.
\end{lemma}

\noindent \textbf{Proof.} Fix some $\xi \in \mathbb{R}^{d}\backslash \{0\}$
and let $j\in \{1,..,d\}$ be such that $\xi _{j}\neq 0$. Since $k$ satisfies
the conditions of Proposition \ref{prop.g}, we have (\ref{k-niu}). When $%
t\rightarrow \infty $ and $\eta \in B(t\xi ,r(t))$, we have $\eta ^{\beta
}/\eta _{j}^{l}=(\xi ^{\beta }/\xi _{j}^{l})+o(1)$, for any multiindex $%
\beta $ with $|\beta |=l$, and $\eta ^{\beta }/\eta _{j}^{l}=o(1)$, for any
multiindex $\beta $ with $|\beta |<l$. Since, the functions $\widehat{\nu }%
_{\beta ,j}$ are bounded, we get 
\begin{eqnarray}
\strokedint_{B(t\xi ,r(t))}\widehat{k}(\eta )d\eta &=&\strokedint_{B(t\xi
,r(t))}\left( \sum_{|\beta |=l}\frac{\xi ^{\beta }}{\xi _{j}^{l}}\widehat{%
\nu }_{\beta ,j}(\eta )\right) d\eta +o(1)  \notag \\
&=&\strokedint_{B(t\xi ,r(t))}\widehat{\nu }_{(\xi ,j)}(\eta )d\eta +o(1),
\label{1-mean-miu}
\end{eqnarray}%
where $\widehat{\nu }_{(\xi ,j)}$ is the Fourier transform of the measure 
\begin{equation}
\nu _{(\xi ,j)}:=\sum_{|\beta |=l}\frac{\xi ^{\beta }}{\xi _{j}^{l}}\nu
_{\beta ,j}.  \label{1-niu-j}
\end{equation}

By Wiener's lemma on $\mathbb{R}^{d}$ (see Lemma \ref{lem.Wien} in the
Appendix) we have 
\begin{equation*}
\lim_{t\rightarrow \infty }\strokedint_{B(t\xi ,r(t))}\widehat{\nu }_{(\xi
,j)}(\eta )d\eta =c\cdot \nu _{(\xi ,j)}(\{0\})\text{.}
\end{equation*}

In particular, the left hand side of (\ref{1-mean-miu}) has a limit and 
\begin{equation}
\lim_{t\rightarrow \infty }\strokedint_{B(t\xi ,r(t))}\widehat{k}(\eta
)d\eta =c\cdot \nu _{(\xi ,j)}(\{0\})\text{.}  \label{1-k-niu-j}
\end{equation}

This proves Lemma \ref{lem.mean}. \hfill $\square $

\begin{remark}
\label{rem.2} We can also get some useful information if we analyse $|%
\widehat{k}|^{2}$ instead of $\widehat{k}$. As in (\ref{1-mean-miu}) we have 
\begin{equation}
\strokedint_{B(t\xi ,r(t))}|\widehat{k}(\eta )|^{2}d\eta =\strokedint%
_{B(t\xi ,r(t))}\left\vert \widehat{\nu }_{(\xi ,j)}(\eta )\right\vert
^{2}d\eta +o(1),  \notag
\end{equation}%
with $\nu _{(\xi ,j)}$ defined by (\ref{1-niu-j}). In this case, Wiener's
theorem on $\mathbb{R}^{d}$ (see Theorem \ref{th.W} in the Appendix) gives
us 
\begin{equation*}
\lim_{t\rightarrow \infty }\strokedint_{B(t\xi ,r(t))}\left\vert \widehat{%
\nu }_{(\xi ,j)}(\eta )\right\vert ^{2}d\eta =c\sum_{\tau \in \mathbb{R}%
^{d}}\left\vert \nu _{(\xi ,j)}(\{\tau \})\right\vert ^{2}\text{.}
\end{equation*}

Therefore, 
\begin{equation}
\lim_{t\rightarrow \infty }\strokedint_{B(t\xi ,r(t))}|\widehat{k}(\eta
)|^{2}d\eta =c\sum_{\tau \in \mathbb{R}^{d}}\left\vert \nu _{(\xi
,j)}(\{\tau \})\right\vert ^{2}\text{.}  \label{k-niu-j}
\end{equation}
\end{remark}

\bigskip

One can prove now Lemma \ref{lem.W}.

\bigskip

\noindent \textbf{Proof of Lemma \ref{lem.W}.} For each $\xi \in \mathbb{R}%
^{d}\backslash \{0\}$ and $t>0$ consider the means

\begin{equation*}
m^{t}(\xi ):=\strokedint_{B(t\xi ,r(t))}\widehat{k}(\eta )d\eta =\strokedint%
_{B(0,r(t))}\widehat{k}(t\xi +\eta )d\eta .
\end{equation*}

Observe that each $m^{t}$ is a bounded continuous function on $\mathbb{R}%
^{d}\backslash \{0\}$ with 
\begin{equation*}
\left\Vert m^{t}\right\Vert _{L^{\infty }}\leq \left\Vert \widehat{k}%
\right\Vert _{L^{\infty }}<\infty ,
\end{equation*}%
(see Proposition \ref{prop.b-inhom}) and that, by Lemma \ref{lem.mean}, $%
m^{t}(\xi )\rightarrow m^{\infty }(\xi )$, for any $\xi \in \mathbb{R}%
^{d}\backslash \{0\}$. In particular, $m^{\infty }$ is a measurable bounded
function on $\mathbb{R}^{d}\backslash \{0\}$.

Denoting the function $\widehat{k}(t\cdot )$ by $\widehat{k}_{t}$, we have 
\begin{equation}
m^{t}(\xi )=\strokedint_{B(0,r(t))}\widehat{k}_{t}(\xi +\eta /t)d\eta .
\label{m-sigma}
\end{equation}

Let $F$ be a function of the form $F=\nabla ^{l}f$, for some scalar Schwartz
function $f$ on $\mathbb{R}^{d}$. Using (\ref{m-sigma}) we obtain 
\begin{equation}
m^{t}(D)F(0)=\strokedint_{B(0,r(t))}\widehat{k}_{t}(D+\eta /t)F(0)d\eta .
\label{m-s-1}
\end{equation}

By a direct computation we have 
\begin{equation*}
\widehat{k}_{t}(D+\eta /t)F(x)=e_{-\eta /t}(x)(\widehat{k}_{t}(D)(e_{\eta
/t}F)(x)),
\end{equation*}%
for any $x\in \mathbb{R}^{d}$, where $e_{\pm \eta /t}(x):=\exp (\pm
i\left\langle \eta /t,x\right\rangle )$. This shows that (\ref{m-s-1}) can
be rewritten as 
\begin{equation}
m^{t}(D)F(0)=\strokedint_{B(0,r(t))}\widehat{k}_{t}(D)(e_{\eta /t}F)(0)d\eta
=I_{t}+II_{t},  \label{m-s-2}
\end{equation}%
where 
\begin{equation*}
I_{t}:=\strokedint_{B(0,r(t))}\widehat{k}_{t}(D)F(0)d\eta =\widehat{k}%
_{t}(D)F(0)\text{,}
\end{equation*}%
and 
\begin{equation*}
II_{t}:=\strokedint_{B(0,r(t))}\widehat{k}_{t}(D)(e_{\eta /t}F-F)(0)d\eta 
\text{.}
\end{equation*}

Fix some multiindex $\alpha $ with $|\alpha |=l$ and let $\psi $ be a
Schwartz function on $\mathbb{R}^{d}$. For any $t>0$, we have%
\begin{eqnarray*}
\nabla ^{\alpha }\sigma _{t}(D)\psi (0) &=&\int_{\mathbb{R}^{d}}(-i\xi
)^{\alpha }\widehat{k}(t\xi )\widehat{\psi }(\xi )d\xi =t^{-l}\int_{\mathbb{R%
}^{d}}(-i\xi )^{\alpha }\widehat{k}(\xi )t^{-d}\widehat{\psi }(\xi /t)d\xi \\
&=&t^{-l}\nabla ^{\alpha }\widehat{k}(D)\psi _{t}(0),
\end{eqnarray*}%
and hence, by the fact that $\widehat{k}(D)$ is bounded on $C_{0}^{l}(%
\mathbb{R}^{d})$, 
\begin{eqnarray*}
|\nabla ^{\alpha }\widehat{k}_{t}(D)\psi (0)| &=&|t^{-l}\nabla ^{\alpha }%
\widehat{k}(D)\psi _{t}(0)|\lesssim t^{-l}\left\Vert \psi _{t}\right\Vert
_{W^{l,\infty }} \\
&\lesssim &\left\Vert \nabla ^{l}\psi \right\Vert _{L^{\infty
}}+\sum_{j=0}^{l-1}t^{j-l}\left\Vert \nabla ^{j}\psi \right\Vert _{L^{\infty
}}.
\end{eqnarray*}

Consequently, one can write 
\begin{equation}
|I_{t}|=|\widehat{k}_{t}(D)F(0)|\lesssim \left\Vert \nabla ^{l}f\right\Vert
_{L^{\infty }}+\sum_{j=0}^{l-1}t^{j-l}\left\Vert \nabla ^{j}f\right\Vert
_{L^{\infty }}\rightarrow \left\Vert \nabla ^{l}f\right\Vert _{L^{\infty }},
\label{I-bound}
\end{equation}%
when $t\rightarrow \infty $.

On the other hand, we have 
\begin{equation*}
II_{t}=\strokedint_{B(0,r(t))}\int_{\mathbb{R}^{d}}\widehat{k}_{t}(\xi )(%
\widehat{F}(\xi -\eta /t)-\widehat{F}(\xi ))d\xi d\eta ,
\end{equation*}%
and hence%
\begin{eqnarray}
|II_{t}| &\leq &\strokedint_{B(0,r(t))}\int_{\mathbb{R}^{d}}|\widehat{k}%
_{t}(\xi )||\widehat{F}(\xi -\eta /t)-\widehat{F}(\xi )|d\xi d\eta  \notag \\
&\leq &\left\Vert \widehat{k}\right\Vert _{L^{\infty }}\int_{\mathbb{R}^{d}}%
\strokedint_{B(0,r(t))}|\widehat{F}(\xi -\eta /t)-\widehat{F}(\xi )|d\eta
d\xi .  \label{II-1}
\end{eqnarray}

Note that, since $F$ is a Schwartz function, $\nabla \widehat{F}$ is also a
Schwartz functions. By the mean value theorem we get 
\begin{equation*}
|\widehat{F}(\xi -\eta /t)-\widehat{F}(\xi )|\leq |\eta /t|\max_{\xi
^{\prime }\in \lbrack \xi -\eta /t,\xi ]}|\nabla \widehat{F}(\xi ^{\prime
})|\lesssim (r(t)/t)\left\langle \xi \right\rangle ^{-d-1},
\end{equation*}%
for any $t$ with $|\xi |/2\geq r(t)/t$. Using this and (\ref{II-1}) we get 
\begin{eqnarray}
|II_{t}| &\lesssim &\int_{|\xi |/2\geq r(t)/t}\strokedint_{B(0,r(t))}|%
\widehat{F}(\xi -\eta /t)-\widehat{F}(\xi )|d\eta d\xi  \notag \\
&&+\int_{|\xi |<2r(t)/t}\strokedint_{B(0,r(t))}|\widehat{F}(\xi -\eta /t)-%
\widehat{F}(\xi )|d\eta d\xi  \notag \\
&\lesssim &\int_{|\xi |/2\geq r(t)/t}(r(t)/t)\left\langle \xi \right\rangle
^{-d-1}d\eta d\xi +\left\Vert \widehat{F}\right\Vert _{L^{\infty
}}\int_{|\xi |<2r(t)/t}d\xi  \notag \\
&\lesssim &(r(t)/t)+(r(t)/t)^{d}\rightarrow 0,  \label{II-bound}
\end{eqnarray}%
when $t\rightarrow \infty $.

Now, (\ref{m-s-2}), (\ref{I-bound}) and (\ref{II-bound}) (and the dominated
convergence theorem) gives us 
\begin{equation*}
|m^{\infty }(D)\nabla ^{l}f(0)|=\lim_{t\rightarrow \infty
}|m^{t}(D)F(0)|\lesssim \left\Vert \nabla ^{l}f\right\Vert _{L^{\infty }},
\end{equation*}
and since $m^{\infty }(D)$ commutes with the translations, 
\begin{equation*}
\left\Vert \nabla ^{l}m^{\infty }(D)f\right\Vert _{L^{\infty }}\lesssim
\left\Vert \nabla ^{l}f\right\Vert _{L^{\infty }},
\end{equation*}
for any Schwartz function $f$ on $\mathbb{R}^{d}$.

Hence, $m^{\infty }(D)$ is a multiplier on $\dot{C}_{0}^{l}(\mathbb{R}^{d})$%
. Notice that $m^{\infty }$ is a homogeneous function of order $0$.
Consequently, by applying Lemma \ref{lem.jfa} we get that $m^{\infty }$ is a
constant function. This proves Lemma \ref{lem.W}. \hfill $\square $

\bigskip

An interesting by-product of Lemma \ref{lem.W} and Remark \ref{rem.2} is
that, for any $k$ with $T_{k}\in M(C_{0}^{l}(\mathbb{R}^{d}))$, the only
contribution to its Wiener number (i.e., (\ref{WB-1})) is inherited from a
discrete measure. More precisely, we have:

\begin{proposition}
\label{prop.discr} Let $k$ be a tempered distribution such that $T_{k}\in
M(C_{0}^{l}(\mathbb{R}^{d}))$. Then, $k$ can be uniquely decomposed as $%
k=k_{d}+k^{\prime }$, where $k_{d}\in \mathcal{M}(\mathbb{R}^{d})$ is a
discrete measure and $k^{\prime }$ is a pseudo-measure such that 
\begin{equation}
\lim_{t\rightarrow \infty }\strokedint_{B(t\xi ,r(t))}|\widehat{k^{\prime }}%
(\eta )|^{2}d\eta =0,  \label{k'0}
\end{equation}%
for any $\xi \in \mathbb{R}^{d}\backslash \{0\}$ and any growth function $r$
with $\lim_{t\rightarrow \infty }(r(t)/t)=0$.
\end{proposition}

\noindent \textbf{Proof.} We show the existence of a discrete measure $\mu
\in \mathcal{M}(\mathbb{R}^{d})$ such that 
\begin{equation}
\lim_{t\rightarrow \infty }\strokedint_{B(t\xi ,r(t))}|\widehat{k}(\eta )-%
\widehat{\mu }(\eta )|^{2}d\eta ,  \label{k0}
\end{equation}%
for any $\xi \in \mathbb{R}^{d}\backslash \{0\}$. (We then show that such $%
\mu $ is unique and set $k_{d}:=\mu $ and $k^{\prime }:=k-k_{d}$.)

Fix some $j\in \{1,...,d\}$ and consider the measures $\nu _{\beta ,j}$ that
correspond to a decomposition of the form (\ref{k-j-dec}). If we denote by $%
\nu ^{(j)}$ the measure $\nu _{\beta _{j},j}$, where $\beta
_{j}:=(0,...,0,l,0,...,0)$ (with $l$ on the $j$-th position), then the left
hand side of (\ref{k-niu-j}) can be written as 
\begin{equation}
\sum_{\tau \in \mathbb{R}^{d}}\left\vert \nu ^{(j)}(\{\tau \})+\sum_{|\beta
|=l,\beta \neq \beta _{j}}\frac{\xi ^{\beta }}{\xi _{j}^{l}}\nu _{\beta
,j}(\{\tau \})\right\vert ^{2},  \label{niuj}
\end{equation}%
for any $\xi \in \mathbb{R}^{d}\backslash \{0\}$ with $\xi _{j}\neq 0$.
However, according to (\ref{k-niu-j}) and Lemma \ref{lem.W} applied to the
symbol\footnote{%
Note that, since $\widehat{k}$ is the symbol of a multiplier on $C_{0}^{l}(%
\mathbb{R}^{d})$, $\overline{\widehat{k}}$ is also the symbol of a
multiplier on $C_{0}^{l}(\mathbb{R}^{d})$. We get that $\sigma :=\widehat{k}%
\overline{\widehat{k}}=|\widehat{k}|^{2}$ is the symbol of a multiplier on $%
C_{0}^{l}(\mathbb{R}^{d})$.} $|\widehat{k}|^{2}$, the quantity in (\ref{niuj}%
) does not depend on $\xi $ (neither on $j$). Hence, we have 
\begin{equation}
\nu _{\beta ,j}(\{\tau \})=0,  \label{niu0}
\end{equation}%
for any $\tau \in \mathbb{R}^{d}$, as long as $\beta \neq \beta _{j}$.

By (\ref{k-niu}), for any $\eta \in \mathbb{R}^{d}\backslash \{0\}$ with $%
\eta _{j}\neq 0$, we have 
\begin{equation*}
\widehat{k}(\eta )-\widehat{\nu }_{d}^{(j)}(\eta )=\widehat{\nu }%
_{c}^{(j)}(\eta )+\sum_{|\beta |\leq l,\beta \neq \beta _{j}}\frac{(i\eta
)^{\beta }}{(i\eta _{j})^{l}}\widehat{\nu }_{\beta ,j}(\eta ),
\end{equation*}
where ${\nu }_{c}$ is the continuous part of $\nu$ (see for instance (\ref%
{miuA-1})).

Now, as in (\ref{k-niu-j}) by Wiener's theorem (Theorem \ref{th.W}), we
obtain that, for any $\xi \in \mathbb{R}^{d}\backslash \{0\}$ with $\xi
_{j}\neq 0$, 
\begin{equation*}
\lim_{t\rightarrow \infty }\strokedint_{B(t\xi ,r(t))}|\widehat{k}(\eta )-%
\widehat{\nu }_{d}^{(j)}(\eta )|^{2}d\eta
\end{equation*}%
equals 
\begin{equation*}
\sum_{\tau \in \mathbb{R}^{d}}\left\vert \nu _{c}^{(j)}(\{\tau
\})+\sum_{|\beta |=l,\beta \neq \beta _{j}}\frac{\xi ^{\beta }}{\xi _{j}^{l}}%
\nu _{\beta ,j}(\{\tau \})\right\vert ^{2},
\end{equation*}%
which by (\ref{niu0}) (and $\nu _{c}^{(j)}(\{\tau \})=0$, for any $\tau \in 
\mathbb{R}^{d}$) is zero. In other words,%
\begin{equation}
\lim_{t\rightarrow \infty }\strokedint_{B(t\xi ,r(t))}|\widehat{k}(\eta )-%
\widehat{\nu }_{d}^{(j)}(\eta )|^{2}d\eta =0\text{.}  \label{kniu-0}
\end{equation}

Let us remark that if (\ref{k0}) holds for some $\xi \in \mathbb{R}%
^{d}\backslash \{0\}$ and some discrete Radon measures $\mu _{1}$, $\mu _{2}$
(in place of $\mu $) then, we have $\mu _{1}=\mu _{2}$. Indeed, by the
triangle inequality 
\begin{equation*}
|\widehat{\mu }_{1}(\eta )-\widehat{\mu }_{2}(\eta )|^{2}\lesssim |\widehat{k%
}(\eta )-\widehat{\mu }_{1}(\eta )|^{2}+|\widehat{k}(\eta )-\widehat{\mu }%
_{2}(\eta )|^{2},
\end{equation*}%
and by (\ref{k0}) for $\mu =\mu _{1}$ and $\mu =\mu _{2}$ we get 
\begin{equation*}
\lim_{t\rightarrow \infty }\strokedint_{B(t\xi ,r(t))}|\widehat{(\mu
_{1}-\mu _{2})}(\eta )|^{2}d\eta =0,
\end{equation*}%
which by Wiener's theorem (Theorem \ref{th.W}) and the discreetness of $\mu
_{1}-\mu _{2}$, shows that $\mu _{1}-\mu _{2}=0$, i.e., we have the
uniqueness of $\mu $ in (\ref{k0})). This fact combined with (\ref{kniu-0})
shows that $\nu _{d}^{(j)}$ are all equal and one can choose $\mu :=\nu
_{d}^{(1)}$. \hfill $\square $

\begin{remark}
\bigskip\ \label{rem.discr}Let $k$ be such that $T_{k}\in M(C_{0}^{l}(%
\mathbb{R}^{d}))$. As in the proof of Proposition \ref{prop.discr} one can
see that for any decomposition (\ref{dec-miu}) the discrete part of the
\textquotedblleft main term\textquotedblright\ $\nu _{\alpha ,\alpha }$ is $%
k_{d}$, for any $\alpha $ with $|\alpha |=l$.
\end{remark}

\begin{remark}
\label{rem.Wdiscr}In general, the Wiener property does not guaranties the
existence of a measure responsible for the whole contribution in (\ref{WB-1}%
). This can be seen by the following easy example. Consider the tempered
distribution $k:=2P_{+}\delta _{0}-\delta _{0}$, where $P_{+}:=1_{[0,\infty
)^{d}}(D)$ is the Riesz projection. Clearly, $|\hat{k}|^{2}$ is constant and
therefore $k\ast \overline{k(-\cdot )}$ has the Wiener property with respect
to $r(t)=\sqrt{t}$. Suppose by contradiction that $k=\mu +k^{\prime }$,
where $\mu $ is a Radon measure and $k^{\prime }$ is as in the statement of
Proposition \ref{prop.discr}. Then, $2P_{+}\delta _{0}=(\mu +\delta
_{0})+k^{\prime }$ and since $\mu +\delta _{0} $ is a Radon measure, $%
2P_{+}\delta _{0}\ast \overline{2P_{+}\delta _{0}(-\cdot )}$ must have the
Wiener property with respect to $r(t)=\sqrt{t}$. However, it is easily seen
that this is not the case by considering the directions of $\xi =(1,..,1)$
and $\xi =(-1,...,-1)$.
\end{remark}

\begin{remark}
Let $k^{1}$ and $k^{2}$ be tempered distributions such that $%
T_{k^{1}},T_{k^{2}}\in M(C_{0}^{l}(\mathbb{R}^{d}))$. A direct consequence
of Proposition \ref{prop.discr} and Wiener's lemma (Lemma \ref{lem.Wien}) is
that 
\begin{equation*}
\lim_{t\rightarrow \infty }\strokedint_{B(t\xi ,\sqrt{t})}\widehat{k^{1}}%
(\eta )\widehat{k^{2}}(\eta )d\eta =c\sum_{\tau \in \mathbb{R}%
^{d}}k_{d}^{1}(\{\tau \})k_{d}^{2}(\{\tau \}),
\end{equation*}%
for any $\xi \in \mathbb{R}^{d}\backslash \{0\}$. Also, as in the case of
measures, we have that $(k^{1}\ast k^{2})_{d}=k_{d}^{1}\ast k_{d}^{2}$.
\end{remark}

We now transfer Lemma \ref{lem.W} from $\mathbb{R}^{d}$ to $\mathbb{T}^{d}$
by the means of Lemma \ref{lem.tor-R} and Lemma \ref{lem.v-psi} below.

\begin{lemma}
\label{lem.tor-R}Let $\sigma \in l^{\infty }(\mathbb{Z}^{d})$ and consider
the distribution $K$ on $\mathbb{T}^{d}$ given by 
\begin{equation*}
K(x)=\sum_{\chi \in \mathbb{Z}^{d}}\sigma (\chi )e^{i\left\langle \chi
,x\right\rangle }.
\end{equation*}

Let $\Phi $ be a Schwartz function on $\mathbb{R}^{d}$ such that $\widehat{%
\Phi }\in C_{c}^{\infty }((B(0,1/4))$, $\widehat{\Phi }\equiv 1$ on $%
B(0,1/8) $ and $\widehat{\Phi }\geq 0$ is radial. Consider the tempered
distribution $K_{e}$ on $\mathbb{R}^{d}$ given by $K_{e}:=\widetilde{K}\Phi $%
, where $\widetilde{K}$ is the extension of $K$ to $\mathbb{R}^{d}$ by
periodicity. Then, there exists a constant $c>0$ (depending on $G$) such
that 
\begin{equation}
\strokedint_{B(t\omega ,r(t))}\widehat{K}_{e}(\xi )d\xi =c\frac{1}{%
|B(t\omega ,r(t))|}\sum_{\chi \in B_{d}(t\omega ,r(t))}\sigma (\chi
)+o_{Q\rightarrow \infty }(1),  \label{tR-0}
\end{equation}%
for any $\omega \in \mathbb{S}^{d-1}$ and any growth function $r$.
\end{lemma}

\noindent \textbf{Proof.} For any $\xi \in \mathbb{R}^{d}$ we have 
\begin{equation}
\widehat{K}_{e}(\xi )=\int_{\mathbb{R}^{d}}\widetilde{K}(x)\Phi
(x)e^{-i\left\langle x,\xi \right\rangle }dx=\sum_{\chi ^{\prime }\in 
\mathbb{Z}^{d}}\sigma (\chi ^{\prime })\widehat{\Phi }(\chi ^{\prime }-\xi ).
\label{tR-1}
\end{equation}

For simplicity we introduce the notation $B^{t}:=B(t\omega ,r(t))$ and $%
B_{d}^{t}:=B_{d}(t\omega ,r(t))$. By (\ref{tR-1}) and the properties of $%
\widehat{\Phi }$, the function $\widehat{K}_{e}$ is bounded. Hence, since $%
B^{t}$ and $\cup _{\chi \in B_{d}^{t}}Q(\chi ,1)$ have asymptotically
equivalent volumes, one can replace the the average on $B^{t}$\ in the left
hand side of (\ref{tR-0}) by an average on that $Q^{t}:=\cup _{\chi \in
B_{d}^{t}}Q(\chi ,1)$.

Fix some integral point $\chi \in Q_{d}$ and let $(\chi
^{j})_{j=1,...,2^{d}} $ be the (integral) extremal points of the cube $%
Q(\chi ,1)$. Since, $\widehat{\Phi }(\chi ^{\prime }-\xi )=0$ for $|n-\xi
|\geq 1/4$, the function $\widehat{\Phi }(\chi ^{\prime }-\cdot )$ is
identically zero on $Q(\chi ,1)$, unless $\chi ^{\prime }$ is one of the
points $\chi ^{j}$. Also, if $j\in \{1,...,2^{d}\}$, the support of $%
\widehat{\Phi }(\chi ^{j}-\cdot )$ in $Q(\chi ,1)$ is the set%
\begin{equation*}
U_{j}(\chi ):=B(\chi ^{j},1/4)\cap Q(\chi ,1),
\end{equation*}%
and we have 
\begin{equation*}
\int_{U_{j}(\chi )}\widehat{\Phi }(\chi ^{j}-\xi )d\xi =\int_{B(0,1/4)\cap
Q(0,1)}\widehat{\Phi }(\xi )d\xi =:c_{1}>0,
\end{equation*}%
for any $\chi $ and $j$. Also, note that for any $\chi $ the sets $%
U_{j}(\chi )$ are pairwise disjoint.

Using these observations and (\ref{tR-1}) one can write 
\begin{equation}
\strokedint_{Q^{t}}\widehat{K}_{e}(\xi )d\xi =\frac{1}{|Q^{t}|}\sum_{\chi
\in B_{d}^{t}}\int_{Q(\chi ,1)}\left( \sum_{\chi ^{\prime }\in \mathbb{Z}%
^{d}}\sigma (\chi ^{\prime })\widehat{\Phi }(\chi ^{\prime }-\xi )\right)
d\xi =\frac{c_{1}}{|Q^{t}|}\sum_{\chi \in B_{d}^{t}}\sum_{j=1}^{2^{d}}\sigma
(\chi ^{j}).  \label{tR-2}
\end{equation}

Now, we can observe that in the right hand side of (\ref{tR-2}) each point $%
\chi ^{j}$ is counted $2^{d}$ times, except for the integral points on the
boundary of $Q^{t}$. However, their number is of smaller order of growth
than $|Q^{t}|$. Taking into account the fact that $\sigma $ is bounded, we
can write 
\begin{equation*}
\frac{c_{1}}{|Q^{t}|}\sum_{\chi \in B_{d}^{t}}\sum_{j=1}^{2^{d}}\sigma (\chi
^{j})=\frac{2^{d}c_{1}}{|B^{t}|}\sum_{\chi \in B_{d}^{t}}\sigma (\chi
)+o_{t\rightarrow \infty }(1),
\end{equation*}%
which together with (\ref{tR-2}) gives us (\ref{tR-0}) with $c:=2^{d}c_{1}$.
\hfill $\square $

\bigskip

Thanks to the special form of the elements of$M(C^{l}(\mathbb{T}^{d}))$ and $%
M(C_{0}^{l}(\mathbb{R}^{d}))$ we have the following transference result:

\begin{lemma}
\label{lem.v-psi} With the same notation as in Lemma \ref{lem.tor-R}, if $%
T_{K}\in M(C^{l}(\mathbb{T}^{d}))$ then $T_{K_{e}}\in M(C_{0}^{l}(\mathbb{R}%
^{d}))$.
\end{lemma}

\noindent \textbf{Proof.} First, it is useful to observe the following
elementary fact:

\begin{claim}
Let $l$ be a nonnegative integer. If $v\in \mathcal{M}^{-l}(\mathbb{T}^{d})$
and $\psi $ is a Schwartz function on $\mathbb{R}^{d}$ then $\widetilde{\nu }%
\psi \in \mathcal{M}^{-l}(\mathbb{R}^{d})$, where is the extension of $v$ by
periodicity to a distribution $\widetilde{\nu }$ on $\mathbb{R}^{d}$
\end{claim}

For $l=0$ the statement can be checked directly. The general case follows
easily by induction on $l$ and a direct application of the Leibniz
differentiation formula.

Coming back to the proof of Lemma \ref{lem.v-psi}, we observe that, by
Proposition \ref{prop.g} we have $\nabla ^{j}K\in \mathcal{M}^{-l}(\mathbb{T}%
^{d})$, for any $j\in \{0,..,l\}$. (Also, by Proposition \ref{prop.b-inhom}, 
$\widehat{K}$ is bounded on $\mathbb{Z}^{d}$.) By the Leibniz formula and
part (i) we have that $\nabla ^{j}K_{e}\in \mathcal{M}^{-l}(\mathbb{R}^{d})$%
, for any $j\in \{0,..,l\}$. Hence, $T_{K_{e}}\in M(C_{0}^{l}(\mathbb{R}%
^{d}))$. (Note that Lemma \ref{lem.v-psi} can also be deduced directly from
Proposition \ref{Bon-Moh-general} for $X(\mathbb{R}^{d})=C_{0}^{l}(\mathbb{R}%
^{d})$.)\hfill $\square $

\bigskip

By combining Lemma \ref{lem.tor-R} and Lemma \ref{lem.v-psi} we can now
establish the Wiener property of the elements of $M(W^{l,\infty }(\mathbb{T}%
^{d}))$.

\begin{lemma}
\label{lem.W'} If $T_{K}\in M(W^{l,\infty }(\mathbb{T}^{d}))$ then $K$ has
the Wiener property on $\mathbb{T}^{d}$ with respect to any growth function $%
r$ for which $\lim_{t\rightarrow \infty }(r(t)/t)=0$.
\end{lemma}

\noindent \textbf{Proof.} We use the notation from Lemma \ref{lem.tor-R}.
Thanks to Remark \ref{rem.g.inf} we have $T_{K}\in M(C^{l}(\mathbb{T}^{d}))$
and by Lemma \ref{lem.v-psi} (ii) we have $T_{K_{e}}\in M(C_{0}^{l}(\mathbb{R%
}^{d}))$. Now, by using Lemma \ref{lem.W} we get that, for any $\omega \in 
\mathbb{S}^{d-1}$, the limit 
\begin{equation*}
\lim_{t\rightarrow \infty }\strokedint_{B(t\omega ,r(t))}\widehat{K}_{e}(\xi
)d\xi ,
\end{equation*}%
(where $\widehat{K}_{e}$ is the Fourier transform of $K_{e}$ with respect to 
$\mathbb{R}^{d}$) exists and is independent of $\xi $. This fact combined
with Lemma \ref{lem.tor-R} show that, for any $\omega \in \mathbb{S}^{d-1}$,
the limit 
\begin{equation*}
\lim_{t\rightarrow \infty }\frac{1}{|B(t\omega ,r(t))|}\sum_{\chi \in
B_{d}(t\omega ,r(t))}\widehat{K}(\chi ),
\end{equation*}%
(where $\widehat{K}$ is the Fourier transform of $K$ with respect to $%
\mathbb{T}^{d}$) exists and is independent of $\xi $. \hfill $\square $

\begin{remark}
\label{rem.discr'}By using Lemma \ref{lem.W'} instead of Lemma \ref{lem.W}
we can obtain the analogue of Proposition \ref{prop.discr} in the case of $%
\mathbb{T}^{d}$.
\end{remark}

\subsection{The Wiener property of multipliers on $W^{l,1}(\mathbb{T}^{d})$}

In this section we show that the Fourier multipliers on $W^{l,1}(\mathbb{T}%
^{d})$ have a Wiener property with respect to some reasonable growth
functions. One natural question is to ask whether new arguments are needed
here. As we have seen in Remarks \ref{rem.g}, \ref{rem.g.inf}, any Fourier
multiplier on $W^{l,\infty }(\mathbb{T}^{d})$ is a Fourier multiplier on $%
W^{l,1}(\mathbb{T}^{d})$, i.e., we have the embedding 
\begin{equation}
M(W^{l,\infty }(\mathbb{T}^{d}))\subseteq M(W^{l,1}(\mathbb{T}^{d})).
\label{inf-1}
\end{equation}

If this embedding would be an equality, that is if the Fourier multipliers
on $W^{l,\infty }(\mathbb{T}^{d})$ and on $W^{l,1}(\mathbb{T}^{d})$ would be
the same, then we could conclude the Wiener property in the case $W^{l,1}(%
\mathbb{T}^{d})$ from the already treated case of $W^{l,\infty }(\mathbb{T}%
^{d})$. This is however not possible due to the following remarkable result
of Bonami and Mohanty in \cite{Bon-Moh}\footnote{%
The proof given in \cite[p. 329-330]{Bon-Moh} is explicitly written in the
case $l=1$. However, it is easily seen that the argument can be adapted to
any $l\geq 1$.}:

\begin{proposition}
\label{prop.Bon-Moh}There exists a symbol $m$ such that $m(D)$ is a bounded
operator from $W^{l,1}(\mathbb{T}^{d})$ to $W^{l,2}(\mathbb{T}^{d})$ (in
particular $m(D)\in M(W^{l,1}(\mathbb{T}^{d}))$) such that the convolution
kernel $k:=m^{\vee }$ does not satisfy (\ref{dec-miu}) on $\mathbb{T}^{d}$.
In particular, we have $m(D)\in M(W^{l,1}(\mathbb{T}^{d}))$ and $m(D)\notin
M(W^{l,\infty }(\mathbb{T}^{d}))$.
\end{proposition}

\bigskip

This shows that embedding (\ref{inf-1}) is strict. In principle, we could
still use the results of Section \ref{sec.inf} if any element $T_{k}$ of $%
M(W^{l,1}(\mathbb{T}^{d}))$ would be an element of $M(W^{l_{k},1}(\mathbb{T}%
^{d}))$, for some nonnegative integer $l_{k}$. However, whether or not this
is true is unknown at this moment.

\bigskip

As in Section \ref{sec.inf} we study first an Euclidean version of the
result. In this case, however, we will deal with a different kind of growth
functions than those considered in Section \ref{sec.inf}. Namely, we
consider here functions of linear growth.

\begin{lemma}
\label{lem.W1} If $T_{k}\in M(W^{l,1}(\mathbb{R}^{d}))$, then $k$ has the
Wiener property on $\mathbb{R}^{d}$ with respect to the growth functions $%
r_{\varepsilon }(t)=\varepsilon t$, for any $\varepsilon \in (0,1/2)$.
\end{lemma}

The proofs here are not completely analogous to those in Section \ref%
{sec.inf}. We do not use here Wiener's lemma for measures (Lemma \ref%
{lem.Wien}) as we did in the proof of Lemma \ref{lem.mean}. Instead, we use
results proved via the Riesz product technique as used in \cite{W}, \cite{KW}
and reused in \cite{CE-F}. In \cite{KW} the Riesz products were used in
order to show that any Fourier symbol of a bounded Fourier multiplier on $%
\dot{W}^{1,1}(\mathbb{R}^{d})$ must be continuous. In \cite{CE-F}, the same
construction, used in a slightly different way, was employed to extend this
result to the case of $\dot{W}^{l,1}(\mathbb{R}^{d})$ (and $\dot{W}%
^{l,\infty }(\mathbb{R}^{d})$), for any integer $l\geq 1$. We use these
methods in the proof of Lemmas \ref{lem.IIs} and \ref{lem.IIs'} below;
then, we use further these results (together with some ideas in \cite{W}) to
show that the corresponding averages of symbols have limit in any direction.
The fact that this limit does not depend on the direction will be a
consequence of the Bonami-Poornima theorem (see Lemma \ref{lem.Bon-Poo}).

The construction of the appropriate Riesz products that will be used next
rests on a dichotomy similar to the one in \cite[Section 2]{KW} (or \cite[%
Section 3]{CE-F}). Suppose $m:\mathbb{R}^{d}\rightarrow \mathbb{C}$ is a
bounded function, continuous on $\mathbb{R}^{d}\backslash \{0\}$, such that
there exists some $\omega _{0}\in \mathbb{R}^{d}\backslash \{0\}$ for which
the limit $\lim_{t\rightarrow \infty }m(t\omega _{0})$ does not exists. We
have at least one of the following two situations:

\bigskip

\begin{itemize}
\item[(i)] \textit{The symmetric case}: there exists $b_{1},b_{2}\in \mathbb{%
C}$, with $b_{1}\neq b_{2}$, and a sequence of positive numbers $%
(t_{n})_{n\geq 1}$, with $\lim_{n\rightarrow \infty }t_{n}=\infty $, and
such that 
\begin{equation}
\lim_{n\rightarrow \infty }m(-t_{2n}\omega _{0})=\lim_{n\rightarrow \infty
}m(t_{2n}\omega _{0})=b_{1}\text{ \ \ and \ \ }\lim_{n\rightarrow \infty
}m(-t_{2n+1}\omega _{0})=\lim_{n\rightarrow \infty }m(t_{2n+1}\omega
_{0})=b_{2};  \label{IIs}
\end{equation}

\medskip

\item[(ii)] \textit{The asymmetric case}: there exists $b_{1},b_{2}\in 
\mathbb{C}$, with $b_{1}\neq b_{2}$, and a sequence of positive numbers $%
(t_{n})_{n\geq 1}$, with $\lim_{n\rightarrow \infty }t_{n}=\infty $, and
such that 
\begin{equation}
\lim_{n\rightarrow \infty }m(-t_{n}\omega _{0})=b_{1}\text{ \ \ \ and \ \ }%
\lim_{n\rightarrow \infty }m(t_{n}\omega _{0})=b_{2}.  \label{IIa}
\end{equation}
\end{itemize}

\bigskip

In what follows we consider show that if a given symbol $m$ satisfies
certain conditions, then $m(D)$ can not be a bounded multiplier on $W^{l,1}(%
\mathbb{R}^{d})$. These conditions correspond to the symmetric and
asymmetric conditions above, (\ref{IIs}), (\ref{IIa}) respectively; however,
they are not synonymous. The conditions (\ref{IIs}), (\ref{IIa}) are
explicitly used only in Lemma \ref{lem.mean1} below.

\subsubsection{A construction corresponding to the symmetric case}

According to the easy Lemma 16 in \cite{CE-F}, for any positive integer $N$
there exists a finite sequence $\left( \sigma _{n}\right) _{1\leq n\leq N}$
in $\left\{ 0,1\right\} $ such that%
\begin{equation}
\left\vert \sum_{n=1}^{N}\frac{\sigma _{k}}{2n}\prod_{j=1}^{n-1}\left( 1+%
\frac{i}{2j}\right) \right\vert \geq \frac{1}{2\pi }\ln N\text{.}
\label{numar1mult}
\end{equation}

\bigskip

Suppose $N\in \mathbb{N}^{\ast }$ is fixed and $\sigma _{1},...,\sigma
_{N}\in \left\{ 0,1\right\} $ are some fixed numbers such that inequality (%
\ref{numar1mult}) holds. Fix also a function $m:\mathbb{R}^{d}\rightarrow 
\mathbb{C}$ and a sequence $\left( a_{n}\right) _{1\leq n\leq N}$ in $%
\mathbb{Z}^{d}$. Concerning $m$, $\left( \sigma _{n}\right) _{1\leq n\leq N}$
and $\left( a_{n}\right) _{1\leq n\leq N}$ we consider the following
properties (that might or not be satisfied): \bigskip

\begin{itemize}
\item[(P1)] for each $n\in \left\{ 1,...,N\right\} $ we have%
\begin{equation*}
\left\vert m\left( \epsilon _{n}a_{n}+\sum_{1\leq j\leq n-1}\epsilon
_{j}a_{j}\right) -\sigma _{n}\right\vert <\frac{1}{4^{N}}\text{,}
\end{equation*}%
for all $\epsilon _{1},...,\epsilon _{n}\in \left\{ -1,0,1\right\} $ with $%
\varepsilon _{n}\neq 0$;

\item[(P2)] for each $n\in \left\{ 1,...,N-1\right\} $ we have%
\begin{equation*}
4\left\vert a_{n}(1)\right\vert <\left\vert a_{n+1}(1)\right\vert ;
\end{equation*}

\item[(P3)] for each $n\in \left\{ 1,...,N\right\} $ we have%
\begin{equation*}
0<\left\vert a_{n}(1)+\sum_{1\leq j\leq n-1}\epsilon _{j}a_{j}(1)\right\vert 
\text{,}
\end{equation*}%
for all $\epsilon _{1},...,\epsilon _{n-1}\in \left\{ -1,0,1\right\} $;

\item[(P4)] for each $n\in \left\{ 1,...,N\right\} $ we have%
\begin{equation*}
\frac{1+\left\vert a_{n}(s)+\sum_{1\leq j\leq n-1}\epsilon
_{j}a_{j}(s)\right\vert ^{l}}{\left\vert a_{n}(1)+\sum_{1\leq j\leq
n-1}\epsilon _{j}a_{j}(1)\right\vert ^{l}}<\frac{1}{4^{N}}\text{,}
\end{equation*}%
for all $\epsilon _{1},...,\epsilon _{n-1}\in \left\{ -1,0,1\right\} $ and
all $s\in \{2,3,...,d\}$.
\end{itemize}

\bigskip

(Here, we denote by $a_{n}(s)$ the $s$-th coordinate of the vector $%
a_{n}=(a_{n}(1),...,a_{n}(d))\in \mathbb{Z}^{d}$.)

\bigskip

\begin{lemma}
\label{lem.IIs} Consider $m\in C_{b}(\mathbb{R}^{d}\backslash \{0\})$.
Suppose that for any positive integer $N$ there exists two finite sequences $%
\left( \sigma _{n}\right) _{1\leq n\leq N}$ , in $\left\{ 0,1\right\} $, and 
$\left( a_{n}\right) _{1\leq n\leq N}$ , in $\mathbb{Z}^{d}$, satisfying the
conditions (\ref{numar1mult}) and (P1)--(P4). Then, $m(D)$ is not a bounded
operator from $W^{l,1}(\mathbb{R}^{d})$ to $\dot{W}^{l,1}(\mathbb{R}^{d})$.
\end{lemma}

\noindent\textbf{Proof.} We follow closely the methods used in \cite[Section
6.1]{CE-F}. We introduce the Euclidean \textquotedblleft Riesz
product\textquotedblright\ (see \cite[(19)]{CE-F}) 
\begin{equation}
R_{N}(x):=-1+\prod_{n=1}^{N}\left( 1+\frac{i}{k}\cos \left\langle
x,a_{n}\right\rangle \right) ,  \label{Riesz}
\end{equation}%
with $x\in \mathbb{R}^{d}$.

Supposing that $m(D)$ is bounded from $W^{l,1}(\mathbb{R}^{d})$ to $\dot{W}%
^{l,1}(\mathbb{R}^{d})$ we get (by duality) that there exists a family $%
(u_{\alpha })_{|\alpha |=l}$ of bounded functions such that 
\begin{equation*}
\partial _{1}^{l}m(D)R_{N}=\sum_{|\alpha |\leq l}\nabla ^{\alpha }u_{\alpha
},
\end{equation*}
in the sense of distributions on $\mathbb{R}^{d}$ and $\left\Vert u_{\alpha
}\right\Vert _{L^{\infty }}\leq C$, for any $\alpha $ with $|\alpha |\leq l$%
, for some constant $C$ independent of $N$.

By using Lemma 17 in \cite{CE-F} one can replace the functions $u_{\alpha }$
by functions that are component-wise $2\pi -$periodic (notice that $\partial
_{1}^{l}m(D)R_{N}$ is already component-wise $2\pi -$periodic). Namely, we
have 
\begin{equation}
\partial _{1}^{l}m(D)R_{N}=\sum_{|\alpha |\leq l}\nabla ^{\alpha }g_{\alpha
},  \label{m-galfa}
\end{equation}%
for some bounded functions $g_{\alpha }$ that are component-wise $2\pi -$%
periodic and $\left\Vert g_{\alpha }\right\Vert _{L^{\infty }}\leq C$, for
any $\alpha $ with $|\alpha |\leq l$.

Now we have from (\ref{m-galfa}) that 
\begin{equation}
\partial _{1}^{l}\left( m(D)R_{N}\ast G_{N}\right) =\sum_{|\alpha |\leq
l}\nabla ^{\alpha }(g_{\alpha }\ast G_{N}),  \label{m-alfa-1}
\end{equation}%
where 
\begin{equation*}
G_{N}(x):=-1+\prod_{n=1}^{N}\left( 1+\cos \left\langle x,a_{n}\right\rangle
\right) ,
\end{equation*}%
and the convolution in (\ref{m-alfa-1}) is taken with respect to $\mathbb{T}%
^{d}$ (note that both sides of (\ref{m-alfa-1}) are component-wise $2\pi -$%
periodic).

One can easily observe that, thanks to the convolution against $G_{N}$, the
spectrum of each term of (\ref{m-alfa-1}) embeds in the set 
\begin{equation}
\Lambda _{N}:=\left\{ \left. \sum_{1\leq j\leq N}\epsilon _{j}a_{j}\text{ }%
\right\vert \text{ }\epsilon _{1},...,\epsilon _{N}\in \left\{
-1,0,1\right\} \right\} \backslash \{0\}\text{,}  \label{m-spec}
\end{equation}%
while $\left\Vert g_{\alpha }\ast G_{N}\right\Vert _{L^{\infty }}\leq 2C$,
due to the fact that $\left\Vert G_{N}\right\Vert _{L^{\infty }}\leq 2$.
Thanks to (P3) we can write 
\begin{equation}
m(D)(R_{N}\ast G_{N})=\sum_{|\alpha |\leq l}\partial _{1}^{-l}\nabla
^{\alpha }(g_{\alpha }\ast G_{N}),  \label{m-alfa-2}
\end{equation}%
in the sense of distributions on $\mathbb{T}^{d}$ and it remains to see that
the left hand side of (\ref{m-alfa-2}) blows up when $N\rightarrow \infty $,
while the right hand side of (\ref{m-alfa-2}) remains bounded.

Indeed, thanks to the special form of the spectrum (see (\ref{m-spec})) and
(P4) (see \cite[(45)]{CE-F}) we get 
\begin{equation*}
\left\Vert \partial _{1}^{-l}\nabla ^{\alpha }(g_{\alpha }\ast
G_{N})\right\Vert _{L^{\infty }}\lesssim C,
\end{equation*}%
for any $\alpha $ with $|\alpha |\leq l$. This shows that the right hand
side of (\ref{m-alfa-2}) is bounded.

\bigskip

However, as in \cite{CE-F} (see (46)--(50) in \cite{CE-F}) the inequality (%
\ref{numar1mult}) together with (P1) lead to 
\begin{equation}
\left\Vert m(D)(R_{N}\ast G_{N})\right\Vert _{L^{\infty }}\gtrsim \ln N,
\label{m-logN}
\end{equation}%
which in view of (\ref{m-alfa-2}) gives a contradiction.

\bigskip

Let us justify (\ref{m-logN}). In order to estimate the left hand side of (%
\ref{m-alfa-2}) we write it in the form 
\begin{equation}
m(D)(R_{N}\ast G_{N})\left( x\right) =\sum_{n=1}^{N}\sum_{\substack{ %
\varepsilon _{1},...,\varepsilon _{n}\in \left\{ -1,0,1\right\}  \\ %
\varepsilon _{n}\neq 0}}m(\epsilon _{1}a_{1}+...\epsilon _{n}a_{n})\left(
\prod_{\epsilon _{j}\neq 0}\frac{i}{4j}\right) e^{i\left\langle x,\epsilon
_{1}a_{1}+...+\epsilon _{n}a_{n}\right\rangle },  \label{multTmRb}
\end{equation}%
and we compare it with the function 
\begin{equation}
Z(x):=\sum_{n=1}^{N}\sum_{\substack{ \epsilon _{1},...,\epsilon _{n}\in
\left\{ -1,0,1\right\}  \\ \epsilon _{n}\neq 0}}\sigma _{n}\left(
\prod_{\epsilon _{j}\neq 0}\frac{i}{4j}\right) e^{i\left\langle
x,\varepsilon _{1}a_{1}+...+\varepsilon _{n}a_{n}\right\rangle }\text{,}
\label{multZ0b}
\end{equation}%
defined on $\mathbb{T}^{d}$. Now, by (\ref{multTmRb}) and (\ref{multZ0b})
the property (P1) gives us 
\begin{equation}
\left\Vert m(D)(R_{N}\ast G_{N})-Z\right\Vert _{L^{\infty }(\mathbb{T}%
^{2})}\leq \left\vert \Lambda _{N}\right\vert 4^{-N}\leq 3^{N}4^{-N}\leq 1%
\text{.}  \label{multrestb}
\end{equation}%
\ 

\bigskip

On the other hand, one can write (see \cite[(48)]{CE-F}) 
\begin{equation*}
Z(x)=\sum_{n=1}^{N}\frac{i\sigma _{n}}{2n}\cos \left\langle
x,a_{n}\right\rangle \prod_{j=1}^{n-1}\left( 1+\frac{i}{2j}\cos \left\langle
x,a_{j}\right\rangle \right) \text{,}
\end{equation*}%
and by (\ref{numar1mult}) we get%
\begin{equation}
\left\Vert Z\right\Vert _{L^{\infty }(\mathbb{T}^{2})}\geq \left\vert
Z(0)\right\vert \geq \frac{1}{2\pi }\ln N\text{.}  \label{multZb}
\end{equation}

The inequalities (\ref{multrestb}), (\ref{multZb}) gives us (\ref{m-logN}%
).\hfill $\square $

\subsubsection{A construction corresponding to the asymmetric case}

Similar to the symmetric case, we consider here the above properties
(P2)--(P4) and (P1') below:

\begin{itemize}
\item[(P1')] for each $n\in \left\{ 1,...,N\right\} $ we have%
\begin{equation*}
\left\vert m\left( \epsilon _{n}a_{n}+\sum_{1\leq j\leq n-1}\epsilon
_{j}a_{j}\right) -\frac{1+\epsilon _{n}}{2}\sigma _{n}\right\vert <\frac{1}{%
4^{N}}\text{,}
\end{equation*}%
for all $\epsilon _{1},...,\epsilon _{n}\in \left\{ -1,0,1\right\} $ with $%
\epsilon _{n}\neq 0$.
\end{itemize}

\bigskip

In this setting we have:

\begin{lemma}
\label{lem.IIs'} Consider $m\in C_{b}(\mathbb{R}^{d}\backslash \{0\})$.
Suppose that for any positive integer $N$ there exists two finite sequences $%
\left( \sigma _{n}\right) _{1\leq n\leq N}$ , in $\left\{ 0,1\right\} $, and 
$\left( a_{n}\right) _{1\leq n\leq N}$ , in $\mathbb{Z}^{d}$, satisfying the
conditions (\ref{numar1mult}) and (P1'), (P2)--(P4). Then, $m(D)$ is not a
bounded operator from $W^{l,1}(\mathbb{R}^{d})$ to $\dot{W}^{l,1}(\mathbb{R}%
^{d})$.
\end{lemma}

\noindent \textbf{Proof}. The proof is almost identical to the proof of
Lemma \ref{lem.IIs}. The test function that we use is again an
\textquotedblleft Euclidean\textquotedblright\ Riesz product as in (\ref%
{Riesz}). The considered differential equations and the needed estimates are
in general the same. The only difference is that, when estimating the left
hand side of (the analogue of) (\ref{m-alfa-2}), we have to compare it with
the function 
\begin{equation*}
Z^{\prime }(x):=\sum_{n=1}^{N}\sum_{\substack{ \epsilon _{1},...,\epsilon
_{n}\in \left\{ -1,0,1\right\}  \\ \epsilon _{n}\neq 0}}\frac{1+\epsilon _{n}%
}{2}\sigma _{n}\left( \prod_{\epsilon _{j}\neq 0}\frac{i}{4j}\right)
e^{i\left\langle x,\epsilon _{1}a_{1}+...+\epsilon _{n}a_{n}\right\rangle }%
\text{.}
\end{equation*}%
defined on $\mathbb{T}^{d}$. The analogue of (\ref{multrestb}) follows
from the property (P1') instead of (P1). We also observe that, since 
\begin{equation*}
\sum_{\epsilon _{n}\in \left\{ -1,1\right\} }\frac{1+\epsilon _{n}}{2}=1,
\end{equation*}%
we have 
\begin{eqnarray*}
Z^{\prime }(0) &=&\sum_{n=1}^{N}\sum_{\epsilon _{1},...,\epsilon _{n-1}\in
\left\{ -1,0,1\right\} }\left( \sum_{\epsilon _{n}\in \left\{ -1,1\right\} }%
\frac{1+\epsilon _{n}}{2}\right) \frac{i}{4k}\sigma _{k}\left( \prod 
_{\substack{ \epsilon _{j}\neq 0  \\ 1\leq j\leq n-1}}\frac{i}{4j}\right) \\
&=&\sum_{n=1}^{N}\sum_{\substack{ \epsilon _{1},...,\epsilon _{n}\in \left\{
-1,0,1\right\}  \\ \epsilon _{n}\neq 0}}\sigma _{k}\left( \prod_{\epsilon
_{j}\neq 0}\frac{i}{4j}\right) ,
\end{eqnarray*}%
which is the same expression as the one for $Z(0)$ in (\ref{multZ0b}). It
follows that we can use the same argument as in (\ref{multZb}). This
concludes the proof of Lemma \ref{lem.IIs'} by the same arguments as in the
proof of Lemma \ref{lem.IIs}.\hfill $\square $

\bigskip

\subsubsection{Proof of Lemma \protect\ref{lem.W1} and its analogue on the
torus}

In what follows we introduce an averaged version of the initial Fourier
symbol. It turns out that the averaged version is a well-behaved function
far away from the origin.

\begin{lemma}
\label{lem.Lip} If $T_{k}\in M(W^{l,1}(\mathbb{R}^{d}))$, then for any fixed 
$\varepsilon \in (0,1/2)$ the function defined by 
\begin{equation}
m_{\varepsilon }(\xi ):=\strokedint_{B(\xi ,\varepsilon |\xi |)}\widehat{k}%
(\eta )d\eta ,  \label{Lip-00}
\end{equation}%
for $\xi \in \mathbb{R}^{d}$, is the symbol of a bounded Fourier multiplier
on $W^{l,1}(\mathbb{R}^{d})$. Moreover, we have 
\begin{equation}
|m_{\varepsilon }(\xi _{1})-m_{\varepsilon }(\xi _{2})|\leq C\frac{|\xi
_{1}-\xi _{2}|}{|\xi _{1}|},  \label{Lip-0}
\end{equation}%
for any $\xi _{1},\xi _{2}\in \mathbb{R}^{d}$ with $|\xi _{1}|\geq 1$ and $%
|\xi _{1}-\xi _{2}|\leq \varepsilon |\xi _{1}|/2$, where $C$ is a constant
depending on $k$ and $\varepsilon $.
\end{lemma}

\bigskip

\noindent \textbf{Proof.} For each $y\in \mathbb{R}^{d}\backslash \{0\}$
consider the linear operator $L_{y}$ that is obtained by composing the
direct rotation $R_{y}$, that transforms $\mathbf{e}_{1}$ into $y/|y|$, with
multiplication by $|y|$ (we have $L_{y}:=|y|R_{y}$). In particular, we have $%
R_{y}(\mathbf{e}_{1})=y/|y|$ and $L_{y}(\mathbf{e}_{1})=y$, for any $y\in 
\mathbb{R}^{d}\backslash \{0\}$. By a change of variables we have%
\begin{equation*}
\strokedint_{B(\mathbf{e}_{1},\varepsilon )}(\widehat{k}\circ L_{y})(\xi )dy=%
\strokedint_{B(\xi ,\varepsilon |\xi |)}\widehat{k}(\eta )d\eta ,
\end{equation*}%
i.e., 
\begin{equation}
m_{\varepsilon }=\strokedint_{Q(\mathbf{e}_{1},\varepsilon )}(\widehat{k}%
\circ L_{y})dy.  \label{Lip-1}
\end{equation}

Note that for any $y\in B(\mathbf{e}_{1},\varepsilon )$ we have $1/2\leq
|y|\leq 3/2$ and hence 
\begin{equation*}
\left\Vert (\widehat{k}\circ L_{y})(D)\right\Vert _{M(W^{l,1}(\mathbb{R}%
^{d}))}=\left\Vert (\widehat{k}\circ |y|R_{y})(D)\right\Vert _{M(W^{l,1}(%
\mathbb{R}^{d}))}\lesssim _{l,d}\left\Vert \widehat{k}(D)\right\Vert
_{M(W^{l,1}(\mathbb{R}^{d}))}.
\end{equation*}

Using this together with (\ref{Lip-1}) one can write 
\begin{equation*}
\left\Vert m_{\varepsilon }(D)\right\Vert _{M(W^{l,1}(\mathbb{R}^{d}))}\leq %
\strokedint_{B(\mathbf{e}_{1},\varepsilon )}\left\Vert (\widehat{k}\circ
L_{y})(D)\right\Vert _{M(W^{l,1}(\mathbb{R}^{d}))}dy\lesssim
_{l,d}\left\Vert \widehat{k}(D)\right\Vert _{M(W^{l,1}(\mathbb{R}^{d}))}.
\end{equation*}

This shows that $m_{\varepsilon }(D)$ is a bounded Fourier multiplier on $%
W^{l,1}(\mathbb{R}^{d})$.

The estimate (\ref{Lip-0}) can be deduced directly from (\ref{Lip-00}).
Indeed, for $\xi _{1},\xi _{2}$ as in the statement we have (with $%
B_{j}:=B(\xi _{j},\varepsilon |\xi _{j}|)$, $j=1,2$) 
\begin{eqnarray*}
|m_{\varepsilon }(\xi _{1})-m_{\varepsilon }(\xi _{2})| &\leq &\frac{1}{%
|B_{1}|}\left\vert \int_{B_{1}}\widehat{k}(\eta )d\eta -\int_{B_{2}}\widehat{%
k}(\eta )d\eta ,\right\vert +\left\vert \frac{1}{|B_{1}|}-\frac{1}{|B_{2}|}%
\right\vert \left\vert \int_{B_{2}}\widehat{k}(\eta )d\eta ,\right\vert \\
&\leq &\left\Vert \widehat{k}\right\Vert _{L^{\infty }}\left( \frac{%
|B_{1}\Delta B_{2}|}{|B_{1}|}+\left\vert \frac{|B_{2}|}{|B_{1}|}%
-1\right\vert \right) \\
&\lesssim &\left\Vert \widehat{k}\right\Vert _{L^{\infty }}\left( \frac{|\xi
_{1}-\xi _{2}|(\varepsilon |\xi _{1}|)^{d-1}}{(\varepsilon |\xi _{1}|)^{d}}%
+\left\vert \frac{(\varepsilon |\xi _{2}|)^{d}}{(\varepsilon |\xi _{1}|)^{d}}%
-1\right\vert \right) ,
\end{eqnarray*}%
and (\ref{Lip-0}) follows by standard computations. (Recall that by
Proposition \ref{prop.b-inhom} the function $\widehat{k}$ is bounded. One
can set $C\lesssim _{d}\left\Vert \widehat{k}\right\Vert _{L^{\infty
}}/\varepsilon $.) \hfill $\square $

\bigskip

\begin{lemma}
\label{lem.mean1} Let $k$ be a tempered distribution on $\mathbb{R}^{d}$
such that $T_{k}\in M(W^{l,1}(\mathbb{R}^{d}))$ and let $m_{\varepsilon }$
be the function defined by (\ref{Lip-00}), where $\varepsilon \in (0,1/2)$.
Then, for each $\xi \in \mathbb{R}^{d}\backslash \{0\}$ the limit $%
\lim_{t\rightarrow \infty }m_{\varepsilon }(t\xi )$ exists and is finite.
\end{lemma}

\noindent \textbf{Proof. }Note that by Lemma \ref{lem.Lip} the function $%
m_{\varepsilon }$ is bounded and hence, if $\lim_{t\rightarrow \infty
}m_{\varepsilon }(t\xi )$ exists, then it is a real number. Therefore, it
remains to show the existence of $\lim_{t\rightarrow \infty }m_{\varepsilon
}(t\xi )$ for all nonzero $\xi $. Suppose by contradiction that there exists
some $\omega _{0}\in \mathbb{S}^{d-1}$ such that the limit $%
\lim_{t\rightarrow \infty }m_{\varepsilon }(t\omega _{0})$ does not exists.
The symmetric and the asymmetric cases will be treated separately. Fix a
large integer $N$ and a sequence $\left( \sigma _{k}\right) _{1\leq k\leq N}$
, in $\left\{ 0,1\right\} $, satisfying the inequality (\ref{numar1mult}).
\bigskip

\textbf{The symmetric case} \medskip

Suppose we have (\ref{IIs}) for $m=m_{\varepsilon }$. By composing $%
m_{\varepsilon }$ with rotations one can suppose without loss of generality
that $\omega _{0}=\mathbf{e}_{1}$. Also, by considering affine
transformations of $m_{\varepsilon }$ ($m_{\varepsilon }\rightarrow a\cdot
m_{\varepsilon }+b$, where $a$, $b$ are constants) we can assume that there
exists a sequence $(t_{n})_{n\geq 1}$, with $t_{n}\rightarrow \infty $, such
that 
\begin{equation*}
\lim_{n\rightarrow \infty }m_{\varepsilon }(\epsilon t_{2n}\mathbf{e}_{1})=0%
\text{\ \ \ \ \ and \ }\ \lim_{n\rightarrow \infty }m_{\varepsilon
}(\epsilon t_{2n+1}\mathbf{e}_{1})=1,
\end{equation*}%
for any $\epsilon \in \left\{ -1,1\right\} $.

By selecting a subsequnce if neccessary, one can suppose that the sequence $%
(t_{n})_{n\geq 1}$ is such that 
\begin{equation}
|m_{\varepsilon }(\epsilon t_{2n}\mathbf{e}_{1})|<\frac{1}{2\cdot 4^{N}}%
\text{\ \ \ \ \ and \ }\ |m_{\varepsilon }(\epsilon t_{2n+1}\mathbf{e}%
_{1})-1|<\frac{1}{2\cdot 4^{N}},  \label{mean1-1}
\end{equation}%
and, moreover, we have $t_{n+1}>4t_{n}>4^{N+1},$ 
\begin{equation}
\frac{\left\vert \sum_{1\leq j\leq n-1}\epsilon _{j}t_{j}\right\vert }{%
\left\vert t_{n}\right\vert }<\min (\frac{1}{C},\frac{\varepsilon }{2})\frac{%
1}{2\cdot 4^{N}}\text{,}  \label{mean1-2}
\end{equation}%
for all $\epsilon _{1},...,\epsilon _{n-1}\in \left\{ -1,0,1\right\} $, for
any $n\geq 1$, where $C$ is the constant in (\ref{Lip-0}). Also, by (\ref%
{Lip-0}) one can suppose that all the numbers $t_{n}$ are integers.

Now, we define the sequence of frequencies $(a_{n})_{1\leq n\leq N}$ as
follows. If $\sigma _{n}=0$, we set $a_{n}:=t_{2n}\mathbf{e}_{1}$ and if $%
\sigma _{n}=1$, we set $a_{n}:=t_{2n+1}\mathbf{e}_{1}$. One can check
directly that the conditions (P2), (P3), (P4) are satisfied. Fix $n\in
\{1,..,N\}$. For the vectors $\xi _{1}:=\epsilon _{n}a_{n}$ and $\xi
_{2}:=\epsilon _{n}a_{n}+\sum_{1\leq j\leq n-1}\epsilon _{j}a_{j}$, for some
fixed $\epsilon _{1},...,\epsilon _{n-1}\in \left\{ -1,0,1\right\} $, $%
\epsilon _{n}\in \left\{ -1,1\right\} $, we have $|\xi _{1}|>4^{N}>1$ and by
(\ref{mean1-2}), $|\xi _{1}-\xi _{2}|\leq \varepsilon |\xi _{1}|/2$. Hence,
by applying (\ref{Lip-0}) (with $\xi _{1}$, $\xi _{2}$ as above) (\ref%
{mean1-1}) and (\ref{mean1-2}) we obtain 
\begin{eqnarray*}
\left\vert m\left( \epsilon _{n}a_{n}+\sum_{1\leq j\leq n-1}\epsilon
_{j}a_{j}\right) -\sigma _{n}\right\vert &\leq &\left\vert m\left( \epsilon
_{n}a_{n}\right) -\sigma _{n}\right\vert +\left\vert m\left( \epsilon
_{n}a_{n}\right) -m\left( \epsilon _{n}a_{n}+\sum_{1\leq j\leq n-1}\epsilon
_{j}a_{j}\right) \right\vert \\
&<&\frac{1}{2\cdot 4^{N}}+C\min (\frac{1}{C},\frac{\varepsilon }{2})\frac{1}{%
2\cdot 4^{N}}\leq \frac{1}{4^{N}},
\end{eqnarray*}%
which gives (P1). By applying Lemma \ref{lem.IIs} we obtain Lemma \ref%
{lem.mean1} in the symmetric case. \bigskip

\textbf{The asymmetric case} \medskip

Suppose we have (\ref{IIa}) for $m=m_{\varepsilon }$. The argument is
similar to the one in the symmetric case. Again, by composing $%
m_{\varepsilon }$ with rotations, one can suppose without loss of generality
that $\omega _{0}=\mathbf{e}_{1}$. Also, by considering affine
transformations of $m_{\varepsilon }$ we can assume that there exists a
sequence $(t_{n})_{n\geq 1}$, with $t_{n}\rightarrow \infty $, such that 
\begin{equation*}
\lim_{n\rightarrow \infty }m_{\varepsilon }(-t_{2n}\mathbf{e}_{1})=0\text{ \
\ and \ \ }\lim_{n\rightarrow \infty }m_{\varepsilon }(\epsilon t_{2n+1}%
\mathbf{e}_{1})=\frac{1+\epsilon }{2},
\end{equation*}%
for any $\epsilon \in \{-1,1\}$. (This can be obtained by relabeling the
terms of $(t_{n})_{n\geq 1}$ in (\ref{IIa}).)

By selecting a subsequence if necessary, one can suppose that the sequence $%
(t_{n})_{n\geq 1}$ is such that 
\begin{equation}
\left\vert m_{\varepsilon }(-t_{2n}\mathbf{e}_{1})\right\vert <\frac{1}{%
2\cdot 4^{N}}\text{ \ \ and \ }\ \left\vert m_{\varepsilon }(\epsilon
t_{2n+1}\mathbf{e}_{1})-\frac{1+\epsilon }{2}\right\vert <\frac{1}{2\cdot
4^{N}},  \label{mean1-1'}
\end{equation}%
for any $\epsilon \in \{-1,1\}$ and any $n\geq 1$. Moreover, we can assume
that $t_{n+1}>4t_{n}>4^{N+1}$, together with the estimate (\ref{mean1-2}),
for any $n\geq 1$, and the fact that all the numbers $t_{n}$ are integers.

We define the frequencies $(a_{n})_{1\leq n\leq N}$ as in the symmetric
case. The conditions (P2), (P3), (P4) are again directly checked. Fix $n\in
\{1,..,N\}$ and $\epsilon _{1},...,\epsilon _{n-1}\in \left\{ -1,0,1\right\} 
$, $\epsilon _{n}\in \left\{ -1,1\right\} $. For the vectors $\xi
_{1}:=\epsilon _{n}a_{n}$ and $\xi _{2}:=\epsilon _{n}a_{n}+\sum_{1\leq
j\leq n-1}\epsilon _{j}a_{j}$, we have $|\xi _{1}|>4^{N}>1$ and by (\ref%
{mean1-2}), $|\xi _{1}-\xi _{2}|\leq \varepsilon |\xi _{1}|/2$. By applying (%
\ref{Lip-0}) (with $\xi _{1}$, $\xi _{2}$ as above), (\ref{mean1-1'}) and (%
\ref{mean1-2}) we obtain (P1'). An application of Lemma \ref{lem.IIs'} gives
us Lemma \ref{lem.mean1} in the asymmetric case. \hfill $\square $

\bigskip

\noindent \textbf{Proof of Lemma \ref{lem.W1}.} As in the proof of Lemma \ref%
{lem.mean} the fact that the directional limit does not depend on the
direction is based on an argument involving homogeneous symbols. More
precisely, thanks to Lemma \ref{lem.mean1} the function 
\begin{equation}
m_{\varepsilon }^{\infty }(\xi ):=\lim_{t\rightarrow \infty }m_{\varepsilon
}(t\xi )=\lim_{t\rightarrow \infty }\strokedint_{B(t\xi ,\varepsilon t|\xi
|)}\widehat{k}(\eta )d\eta ,  \label{media}
\end{equation}%
is well defined and $0-$homogeneous on $\mathbb{R}^{d}$. Also, since $%
m_{\varepsilon }(D)\in M(W^{l,1}(\mathbb{R}^{d}))$ (see Lemma \ref{lem.Lip}%
), by a standard dilation argument (using the $0-$homogeneity of $%
m_{\varepsilon }^{\infty }$) we get $m_{\varepsilon }^{\infty }(D)$ is a
bounded Fourier multiplier on the homogeneous space $\dot{W}^{l,1}(\mathbb{R}%
^{d})$. Now the Bonami-Poornima result (see Lemma \ref{lem.Bon-Poo}) implies
that $m_{\varepsilon }^{\infty }$ is a constant function.\hfill $\square $

\begin{remark}
\label{rem.1-2}By the methods used in the proof of Lemma \ref{lem.W1} one
can prove an analogue of Lemma \ref{lem.W} (in the case $T_{k}\in
M(C_{0}^{l}(\mathbb{R}^{d}))$) for the linear growth functions $%
r_{\varepsilon }$. In this case, instead of proceding by duality, we
directly study the boundedness in $W^{l,\infty }(\mathbb{R}^{d})$ of the
expressions $m(D)h_{N}$, where formally $h_{N}:=\partial _{1}^{-l}R_{N}$
(see (\ref{Riesz})). By the methods in \cite[Section 5]{CE-F} we obtain the
analogues of Lemmas \ref{lem.IIs} and \ref{lem.IIs'}. Further we use the
same arguments as in the proof Lemma \ref{lem.W1} with Lemma \ref%
{lem.Bon-Poo} replaced by Lemma \ref{lem.jfa}. Note that, as in Proposition %
\ref{prop.discr}, for any $k$ with $T_{k}\in M(C_{0}^{l}(\mathbb{R}^{d}))$
we obtain a unique decomposition $k=k_{d}+k^{\prime }$, where $k_{d}$ is a
finite discrete measure and $k^{\prime }$ is a pseudomeasure with 
\begin{equation*}
\lim_{t\rightarrow \infty }\strokedint_{B(t\xi ,\varepsilon t|\xi |)}|%
\widehat{k^{\prime }}(\eta )|^{2}d\eta =0,
\end{equation*}%
for any $\xi \in \mathbb{R}^{d}\backslash \{0\}$. As in Remark \ref%
{rem.discr}, for any decomposition (\ref{dec-miu}) the discrete part of the
\textquotedblleft main term\textquotedblright\ $\nu _{\alpha ,\alpha }$ is $%
k_{d}$, for any $\alpha $ with $|\alpha |=l$. This shows in particular that $%
k_{d}$ and $k^{\prime }$ above are the same regardless of the fact that we
work with linear or sublinear growth functions.
\end{remark}

\begin{remark}
\label{rem.1-3}It is not known at this moment whether one have a
decomposition similar to the one in Proposition \ref{prop.discr} in the case
of the kernels $k$ with $T_{k}\in M(W^{l,1}(\mathbb{R}^{d}))$.
\end{remark}

We can now deal with the corresponding Wiener property on the torus. We use
the same approach as in Section \ref{sec.inf}, namely from a given kernel $K$
on $\mathbb{T}^{d}$ we pass to its Euclidean version $K_{e}:=\widetilde{K}%
\Phi $, where $\widetilde{K}$ and $\Phi $ are as in Lemma \ref{lem.tor-R}.
However, since for the multipliers on $W^{l,1}(\mathbb{R}^{d})$ or $W^{l,1}(%
\mathbb{T}^{d})$ we do not have direct descriptions as in Proposition \ref%
{prop.g}, we can not transfer Lemma \ref{lem.W1} to the torus by using Lemma %
\ref{lem.v-psi}. We use instead the following slight generalisation of
Theorem 3.4 in \cite{Bon-Moh} (see Proposition \ref{Bon-Moh-general} for $%
X=W^{l,1}$):

\begin{proposition}
\label{Bon-Moh-trans}Suppose $K$ is as in Lemma \ref{lem.tor-R} such that $%
T_{K}\in M(W^{l,1}(\mathbb{T}^{d}))$ and $\psi $ is a Schwartz function on $%
\mathbb{R}^{d}$. Then, $T_{\widetilde{K}\psi }\in M(W^{l,1}(\mathbb{R}^{d}))$%
. In particular, with the notation of Lemma \ref{lem.tor-R} ($K_{e}:=%
\widetilde{K}\Phi $), we have $T_{K_{e}}\in M(W^{l,1}(\mathbb{R}^{d}))$.
\end{proposition}

Now, as in the proof of Lemma \ref{lem.W'}, we can combine Lemma \ref%
{lem.tor-R}, Lemma \ref{lem.W1} and Lemma \ref{lem.tor-R} to obtain the
torus version of Proposition \ref{Bon-Moh-trans}:

\begin{lemma}
\label{lem.W'-1} If $T_{K}\in M(W^{l,1}(\mathbb{T}^{d}))$, then $K$ has the
Wiener property on $\mathbb{T}^{d}$ with respect to the growth functions $%
r_{\varepsilon }(t)=\varepsilon t$, for any $\varepsilon \in (0,1/2)$.
\end{lemma}

\section{Nonexistence of bounded projections onto $A(D)-$free spaces}

\label{sec.proj}

Let $X$ be a Banach space of (scalar) distributions on $\mathbb{T}^{d}$. We
say that $X$ is a \textit{Wiener admissible space }if

\begin{itemize}
\item[(i)] $X$ is translation invariant, i.e., for any $f\in X$ we have $%
f(\cdot +\theta )\in X$, for any $\theta \in \mathbb{T}^{d}$, and 
\begin{equation*}
\left\Vert f(\cdot +\theta )\right\Vert _{X}\lesssim \left\Vert f\right\Vert
_{X},
\end{equation*}%
where the implicit constant does not depend on $\theta $.

\item[(ii)] Any kernel $k$ with $T_{k}\in M(X)$ is a pseudomeasure, i.e., $%
\widehat{k}\in \ell ^{\infty }(\mathbb{Z}^{d})$;

\item[(iii)] There exists a family of growth functions $(r_{\varepsilon
})_{\varepsilon \in (0,1)}$ satisfying (\ref{r-epsilon}) such that any
kernel $k$ with $T_{k}\in M(X)$ has the Wiener property with respect to $%
r_{\varepsilon }$, for any $\varepsilon \in (0,\varepsilon_{0})$ (for some $%
\varepsilon_{0}>0$).
\end{itemize}

For such spaces we have the following result.

\begin{lemma}
\label{lem.X} Let $d,N\geq 2$ be integers and consider a matrix function $A:%
\mathbb{R}^{d}\rightarrow M_{N}(\mathbb{C)}$ satisfying (A1)--(A3) (from the
statement of Theorem \ref{th.A}).Suppose $X$ is a Wiener admissible space on 
$\mathbb{T}^{d}$. Then, $W_{A}(X)$ is not complemented in $X^{N}$.
\end{lemma}

\noindent \textbf{Proof.} In the proof we prefer to write the vectors on
columns. Suppose by contradiction that there exists a bounded projection $%
P:(X^{N})^{\dag }\rightarrow (W_{A}(X))^{\dag }$ that is onto on $%
(W_{A}(X))^{\dag }$. Let $\widetilde{P}$ be the averaged version of $P$,
namely, 
\begin{equation*}
\widetilde{P}:=\int_{\mathbb{T}^{d}}(\mathcal{T}_{-\theta }\circ P\circ 
\mathcal{T}_{\theta })d\theta ,
\end{equation*}%
where $\mathcal{T}_{\theta }$ is the translation by the vector $\theta $: $%
\mathcal{T}_{\theta }F(x)=F(x+\theta )$, for any $F\in (X^{N})^{\dag }$.

Note that the operator $\widetilde{P}:(X^{N})^{\dag }\rightarrow
(W_{A}(X))^{\dag }$ is well-defined, bounded (by (i)) and an onto projection
on $(W_{A}(X))^{\dag }$. Moreover, since $\widetilde{P}$ is commuting with
the translations, there exists a matrix $K$ of distributions on $\mathbb{T}%
^{d}$ such that we have 
\begin{equation}
(\widetilde{P}F)^{\wedge }(\chi )=\widehat{K}(\chi )\widehat{F}(\chi ),
\label{matrix-1}
\end{equation}%
on $\mathbb{Z}^{d}$, for any $F=(f_{1},...,f_{d})^{\dag }\in (X^{N})^{\dag }$
(and $\widehat{F}=(\widehat{f}_{1},...,\widehat{f}_{d})^{\dag }$), where the
entries of $\widehat{K}$ are the Fourier transforms of the corresponding
entries in $K$. Since, $\widetilde{P}$ is bounded on $(X^{d})^{\dag }$, by
testing against elements of the form $\mathbf{e}_{j}^{\dag }f\in
(X^{N})^{\dag }$, for $j\in \{1,...,d\}$, we obtain that each entry of $%
\widehat{K}(D)$ belongs to $M(X)$.

\bigskip

On the other hand, the fact that $\widetilde{P}$ is an onto projection on $%
(W_{A}(X))^{\dag }$ imposes some algebraic conditions on the matrix $%
\widehat{K}(\chi )$ in (\ref{matrix-1}). First, the fact that $\widetilde{P}$
is a projection gives us that 
\begin{equation}
A\widehat{K}=0.  \label{alg-1}
\end{equation}

Also, since $\widetilde{P}$ is an onto projection on $(W_{A}(X))^{\dag }$ we
get 
\begin{equation}
\widehat{K}(\chi )z_{\chi }=z_{\chi },  \label{alg-2}
\end{equation}%
for any vector $z_{n}\in KerA(\chi )$, for any $\chi \in \mathbb{Z}^{d}$. By
setting $z_{\chi }=(I-A^{+}(\chi )A(\chi ))w$ (for some $w\in \mathbb{C}^{N}$%
) where $A^{+}(\chi )$ is the Moore-Penrose pseudoinverse of $A(\chi )$, (%
\ref{alg-2}) gives us 
\begin{equation}
\widehat{K}(I-A^{+}A)=I-A^{+}A.  \label{alg-3}
\end{equation}

Let $(r_{\varepsilon })_{\varepsilon \in (0,1)}$ be a family of growth
functions as in (iii). Fix some unit vector $\omega \in \mathbb{S}^{d-1}$ a
parameter $\varepsilon >0$ and for any $t>0$, consider the discrete balls $%
B^{t}:=B_{d}(t\omega ,r_{\varepsilon }(t))$ (that depends on $\omega $ and $%
\varepsilon $). We have 
\begin{equation}
\lim_{t\rightarrow \infty }\frac{1}{|B^{t}|}\sum_{\chi \in B^{t}}\widehat{K}%
(\chi )=\Gamma _{\varepsilon },  \label{alg-4}
\end{equation}%
where $\Gamma _{\varepsilon }$ is an $N\times N$ matrix, independent of $%
\omega $.

Since $A$ is continuous on $\mathbb{S}^{d-1}$ and $0-$homogeneous (see
(A1)), we have that 
\begin{equation*}
A(\chi )=A(t\omega +\eta _{\chi })=A(\omega +\eta _{\chi }/t)=A(\omega
)+o_{\varepsilon \rightarrow 0}(1),
\end{equation*}%
uniformly in $\chi \in Q^{t}$, for $t$ sufficiently large (here $\eta _{\xi
}:=\xi -t\omega $ and hence, $|\eta _{\xi }|\leq r_{\varepsilon }(t)$).
Using this, (\ref{alg-1}), (\ref{alg-4}) and the fact that $\widehat{K}$ is
bounded (see (ii)) we get%
\begin{equation}
\lim_{t\rightarrow \infty }\frac{1}{|B^{t}|}\sum_{\chi \in B^{t}}A(\chi )%
\widehat{K}(\chi )=A(\omega )\Gamma _{\varepsilon }+o_{\varepsilon
\rightarrow 0}(1).  \label{A-K-0}
\end{equation}

Note that by (ii) the entries of $\Gamma _{\varepsilon }$ are uniformly
bounded. Hence, we can find a sequence $(\varepsilon _{n})_{n\geq 1}$ of
positive numbers with $\varepsilon _{n}\rightarrow 0$ and $\Gamma
_{\varepsilon _{n}}\rightarrow \Gamma $, for some $N\times N$ constant
matrix $\Gamma $, independent on $\omega$ and $\varepsilon$. Now, (\ref%
{A-K-0}), for $\varepsilon =\varepsilon _{n}$, together with (\ref{alg-1})
(and letting $n\rightarrow \infty $) gives us 
\begin{equation}
A(\omega )\Gamma =0,  \label{oo-0}
\end{equation}%
for any $\omega \in \mathbb{S}^{d-1}$.

Let $c_{0}=1$, $c_{1}$, ..., $c_{N}:\mathbb{R}^{d}\rightarrow \mathbb{R}$ be
the coefficients in the characteristic polynomial of $AA^{\ast }$, i.e., 
\begin{equation*}
\det (AA^{\ast }-xI)=(-1)^{N}\sum_{j=0}^{N}c_{j}x^{N-j}.
\end{equation*}

If $s\in \{0,1,...,N\}$ is the largest integer for which we have $c_{s}\neq
0 $, then (see \cite[Theorem 3]{Dec}) 
\begin{equation}
A^{+}=-\frac{1}{c_{s}}A^{\ast }\sum_{j=0}^{s-1}c_{j}(AA^{\ast })^{(s-1)-j}.
\label{alg-ps}
\end{equation}

By (iii) one can find a non-empty open subset $U$ of $\mathbb{S}^{d-1}$ such
that, $A$ is non-invertible on $U$ and $s$ is constant on $U$. By (\ref%
{alg-ps}) we get that $A^{+}$ (and consequently $I-A^{+}A$) is continuous on 
$U$.

As in (\ref{oo-0}) we get from (\ref{alg-3}) that 
\begin{equation}
\Gamma (I-A^{+}(\omega )A(\omega ))=I-A^{+}(\omega )A(\omega ),  \label{oo-1}
\end{equation}%
for any $\omega \in U$. However, (\ref{oo-0}) and condition (A2) implies
that $\Gamma \equiv 0$. Combining this with (\ref{oo-1}), we obtain that $%
A^{+}A=I$, on $U$, i.e., $A$ is invertible on $U$, which contradicts our
choice of $U$. \hfill $\square $

\bigskip

We can prove now Theorem \ref{th.A}.

\noindent\textbf{Proof of Theorem \ref{th.A}.} According to Proposition \ref%
{prop.b-inhom}, Lemma \ref{lem.W'} and Lemma \ref{lem.W'-1} the spaces $%
C^{l}(\mathbb{T}^{d})$, $W^{l,\infty }(\mathbb{T}^{d})$ and $W^{l,1}(\mathbb{%
T}^{d})$ are Wiener admissible (see Remark \ref{rem.r-ep} for the choice of
the family $(r_{\varepsilon })_{\varepsilon \in (0,\varepsilon_0)}$
corresponding to the two cases). It remains to apply Lemma \ref{lem.X} to
these spaces. \hfill $\square $

\begin{remark}
\label{rem.X}In the case where $X=C(\mathbb{T}^{d})$, $L^{\infty }(\mathbb{T}%
^{d})$ or $L^{1}(\mathbb{T}^{d})$, all the Fourier multipliers on $X$ are
given by convolution with Radon measures. Hence, such a space $X$ satisfies
the hypotheses of Theorem \ref{th.A} and we get directly (see Remark \ref%
{rem.A}) Theorem \ref{th.Henk'} in the case $l=0$ and the fact that $%
G_{1}(L^{1}(\mathbb{T}^{d}))$ is not complemented in $(L^{1}(\mathbb{T}%
^{d}))^{d}$ (for a stronger statement see \cite[Corollary 2]{Ki}).
\end{remark}

Theorem \ref{th.A} can be used to easily show that gradient spaces of higher
order are not complemented in gradient spaces of lower order. To be more
precise let us introduce some notation. Given a distribution $f$ on $\mathbb{%
T}^{d}$ and an integer $s\geq 0$, the $s-$gradient of $f$ is the vector
valued distribution $\nabla ^{s}f=(\nabla ^{\alpha }f)_{|\alpha |=s}$ (by
convention $\nabla ^{0}f=f$), where the elements $\nabla ^{\alpha }f$ are
listed in the lexicographic order given by $\alpha $ (note that $\nabla
^{s}f $ has $d^{s}$ components). Suppose $Y(\mathbb{T}^{d})$ is one of the
space $C^{l_{0}}(\mathbb{T}^{d})$, $W^{l_{0},\infty }(\mathbb{T}^{d})$ or $%
W^{l_{0},1}(\mathbb{T}^{d})$ for some integer $l_{0}\geq 0$. For any integer 
$s\geq 0$, we define the space of the $s-$gradients of $Y(\mathbb{T}^{d})$,
as 
\begin{equation*}
G_{s}(Y(\mathbb{T}^{d})):=\{\nabla ^{s}f\in (Y(\mathbb{T}^{d}))^{d^{s}}\mid f%
\text{ distribution on }\mathbb{T}^{d}\}\subset (Y(\mathbb{T}^{d}))^{d^{s}},
\end{equation*}%
where the norm is induced by the norm in $(Y(\mathbb{T}^{d}))^{d^{s}}$ (by
convention $G_{0}(Y(\mathbb{T}^{d}))=Y(\mathbb{T}^{d})$). Also, for any
integer $l\geq 0$, by $Y^{(l)}(\mathbb{T}^{d})$ we denote the space $%
C^{l_{0}+l}(\mathbb{T}^{d})$, $W^{l_{0}+l,\infty }(\mathbb{T}^{d})$ or $%
W^{l_{0}+l,1}(\mathbb{T}^{d})$ respectively. Note that for any integers $%
s,l,j\geq 0$ with $j<s$, the space $G_{s}(Y^{(l)}(\mathbb{T}^{d}))$ is a
closed subspace of $(G_{j}(Y^{(l)}(\mathbb{T}^{d})))^{d^{s-j}}$.

\begin{theorem}
\label{th.GC} Suppose $d\geq 2$. Let $Y(\mathbb{T}^{d})$ be one of the space 
$C^(\mathbb{T}^{d})$, $L^{\infty }(\mathbb{T}^{d})$ or $L^{1}(\mathbb{T}%
^{d}) $. For any nonnegative integers $l,s,j$ with $j<s$, the space $%
G_{s}(Y^{(l)}(\mathbb{T}^{d}))$ is noncomplemented in $(G_{j}(Y^{(l)}(%
\mathbb{T}^{d})))^{d^{s-j}}$.
\end{theorem}

\noindent \textbf{Proof.} To ease the notation we denote the space $Y(%
\mathbb{T}^{d})$ by $Y$. Several reductions are in order. First, we observe
that the general statement of Theorem \ref{th.GC} can be deduced from the
case $l=0$. This follows from the commutative diagram 
\begin{equation*}
\begin{tikzcd} (G_{j}(Y^{(l)}))^{d^{s-j}} \arrow[r,"P"] & G_{s}(Y^{(l)})
\arrow{d}{\gamma_{l}} \\ (G_{j}(G_{l}(Y))^{d^{s-j}} \arrow[u, "h_{l}"]
\arrow[r, "P_{1}"]& G_{s}(G_{l}(Y)) \end{tikzcd}
\end{equation*}%
and the equalities of spaces 
\begin{equation*}
(G_{j}(G_{l}(Y)))^{d^{s-j}}=(G_{j+l}(Y))^{d^{s-j}},
\end{equation*}%
and 
\begin{equation*}
G_{s}(G_{l}(Y))=G_{s+l}(Y).
\end{equation*}

Indeed, the existence of a bounded projection $P$ implies the existence of a
bounded projection $P_{1}:=\gamma _{l}\circ P\circ h_{l}$, where $h_{l}$ and 
$\gamma _{l}$ are defined as follows. The mapping $\gamma _{l}$ is the
operator $\nabla ^{l}$ acting on each scalar coordinate function. For a
given element $F\in (G_{j}(G_{l}(Y)))^{d^{s-j}}$, $h_{l}(F)$ is the unique
vector valued function of mean zero on $\mathbb{T}^{d}$ that belongs to $%
(G_{j}(Y^{(l)}))^{d^{s-j}}$ and $\nabla ^{l}h_{l}(F)=F$ (again, $\nabla ^{l}$
is acting on each scalar coordinate function).

To disprove the existence of such $P_{1}$, it suffices to treat only the
particular case where instead of a general $j$ ($<s$) we have $s-1$. This
can be seen from the commutative diagram: 
\begin{equation*}
\begin{tikzcd}[column sep=small] (G_{j}(Y))^{d^{s-j}} \arrow[rr, "P_{1}'"] &
& G_{s}(Y)\\ & (G_{s-1}(Y))^{d} \arrow[ul, hook, "\iota"] \arrow[ur,
"P_{2}"] & \end{tikzcd}
\end{equation*}%
where $\iota $ is the canonical inclusion. The equality of spaces 
\begin{equation*}
G_{s}(Y)=G_{1}(G_{s-1}(Y)),
\end{equation*}%
enable us to use the following commutative diagram to reduce further the
problem. With $h_{s-1}^{\prime }$, $\gamma _{s-1}^{\prime }$ defined in a
natural way (see the definition of $h_{l}$, $\gamma _{l}$ above) we have 
\begin{equation*}
\begin{tikzcd} (G_{s-1}(Y))^{d} \arrow[r,"P_{2}"] & G_{1}(G_{s-1}(Y))
\arrow{d}{h_{s-1}'} \\ (Y^{(s-1)})^{d} \arrow[u, "\gamma_{s-1}'"] \arrow[r,
"P_{3}"]& G_{1}(Y^{(s-1)}) \end{tikzcd}
\end{equation*}%
and it remains to prove that such a bounded projection $P_{3}$ does not
exist. However, this follows directly from Theorem \ref{th.A}. \hfill $%
\square $

\bigskip

Let us mention that, despite the fact that Henkin's
argument (see \cite{Henk}) is written in the case of $C^{l}(\mathbb{S}^{2})$%
, one can use as well other \textquotedblleft reasonable\textquotedblright  \ spaces $X(\mathbb{S}^{2})$ as for instance, when $X=W^{l,1}$. (We need the
norm of $X$ to be invariant under the action of the orthogonal group and
that $X$ can be reduced to its \textquotedblleft flat\textquotedblright\
version, on $\mathbb{R}^{d}$.) This is due to the fact that the main part of
the argument is algebraic (it is not strongly related to the function space
we use). For simplicity let us work on the $2-$dimensional sphere $\mathbb{S}%
^{2}$. Assume that $P$ is a bounded onto projection from the tangent vector
fields on $\mathbb{S}^{2}$ to the gradients of functions on $\mathbb{S}^{2}$%
. By averaging against the orthogonal group (which is compact) we obtain a
projection $\widetilde{P}$ that is invariant under the action of this group
onto on the gradient space. However, it turns out that there exists only one
projection $\widetilde{P}$ with this property, for which Henkin found the
exact formula: for any vector field $F$ in $X(\mathbb{S}^{2})$ tangent to $%
\mathbb{S}^{2}$ we have

\begin{equation*}
\widetilde{P}F(x)=\nabla \int_{\mathbb{S}^{2}}ctg\left( \frac{\rho (x,y)}{2}%
\right) \left\langle F(y),\mathbf{\tau }_{y,x}\right\rangle d\sigma (y).
\end{equation*}

Here, $d\sigma $ is the normalized \textquotedblleft
surface\textquotedblright\ measure on $\mathbb{S}^{2}$, $\rho (x,y)$ is
the geodesic distance on $\mathbb{S}^{d-1}$ between $x$ and $y$, and, $%
\mathbf{\tau }_{y,x}$ is the unit vector tangent to $\mathbb{S}^{d-1}$ in
the point $y$ that points to $x$ on the small arc joining $y$ and $x$ in $%
\mathbb{S}^{d-1}$.

It remains to show that $\widetilde{P}$ is not bounded. Henkin does this by
constructing a counterexample directly for the operator $\widetilde{P}$, in
the case of the space $X=C^{l}$. Nevertheless, it is convenient to replace $%
\widetilde{P}$ with an Euclidean version as follows. We consider $F$
supported in small caps around a fixed point. By a dilation and limiting
argument we get that the operator 
\begin{equation*}
\widetilde{P}_{e}F(x):=\nabla \int_{\mathbb{R}^{2}}\frac{1}{|x-y|}%
\left\langle F(y),\frac{x-y}{|x-y|}\right\rangle dy,
\end{equation*}%
must be bounded on $X(\mathbb{R}^{2})$ (where $F$ is a vector field on $%
\mathbb{R}^{2}$). This implies that the double Riesz transforms are bounded
on $X(\mathbb{R}^{2})$. In the case of $X=W^{l,\infty }$ (alternatively $%
C^{l}$) or $X=W^{l,1}$ this can be disproved by using Lemma \ref{lem.jfa} or
Lemma \ref{lem.Bon-Poo}.

The drawback of this method is that it does not fit to the more general $%
A(D) $ operators covered by Theorem \ref{th.A}.

\section*{Appendix}

\subsection{Wiener's Theorem for the singularities of measures}

For the sake of completeness we give here a multidimensional version of
Wiener's theorem on $\mathbb{R}^{d}$. For a multidimensional version on $%
\mathbb{T}^{d}$ see \cite[Lemma 9]{KW-1}.

As in \cite[p. 44]{K} we observe that, any measure $\mu \in \mathcal{M}(%
\mathbb{R}^{d})$ can be uniquely decomposed as 
\begin{equation}
\mu =\mu _{c}+\mu _{d},  \label{miuA-1}
\end{equation}%
where $\mu _{c}$, $\mu _{d}$ are the continuous respectively discrete part
of $\mu $. Indeed, define the set of \textquotedblleft
singularities\textquotedblright\ $S:=\{\tau \in \mathbb{R}^{d}$ $|$ $\mu
(\{\tau \})\neq 0\}$. We have 
\begin{equation*}
\sup_{\substack{ J\subseteq S  \\ |J|<\infty }}\sum_{\tau \in J}\left\vert
\mu (\{\tau \})\right\vert \leq |\mu |(\mathbb{R}^{d})<\infty .
\end{equation*}

It follows that $S$ is countable, in particular $\mu-$measurable. We can
define $\mu _{d}$ by the equality $\mu _{d}(E)=\mu (E\cap S)$, for any $\mu
- $measurable set $E\subseteq \mathbb{R}^{d}$. If $(\tau _{j})_{j\geq 1}$ is
the sequence of the elements of $S$, we have 
\begin{equation*}
\mu _{d}=\sum_{j=1}^{\infty }a_{j}\delta _{\tau _{j}},
\end{equation*}
where $a_{j}:=\mu (\{\tau _{j}\})$, for any $j\geq 1$. One can put now $\mu
_{c}:=\mu -\mu _{d}$ and check directly that $\mu _{c}$ is a continuous
measure ($\mu _{c}(\{\tau \})=0$ for any $\tau \in \mathbb{R}^{d}$). Also,
it can be seen that the decomposition (\ref{miuA-1}) is unique.

If $\mu ^{\sharp }\in \mathcal{M}(\mathbb{R}^{d})$, with $\mu ^{\sharp }(E):=%
\overline{\mu (-E)}$, for any $\mu -$measurable set $E\subseteq \mathbb{R}
^{d}$, we have 
\begin{equation*}
\mu _{d}^{\sharp}=\sum_{j=1}^{\infty }\overline{a_{j}}\delta _{-\tau _{j}}.
\end{equation*}

Since $(\mu \ast \mu ^{\sharp })_{d}=\mu _{d}\ast \mu _{d}^{\sharp }$ , as
in \cite[Lemma 7.13]{K} we get by a direct computation that 
\begin{equation}
\mu \ast \mu ^{\sharp }(\{0\})=\mu _{d}\ast \mu _{d}^{\sharp
}(\{0\})=\sum_{j=1}^{\infty }|a_{j}|^{2}.  \label{miuA-conv}
\end{equation}

We have the following general result that we will refer to as Wiener's lemma:

\begin{lemma}
\label{lem.Wien}Let $\nu \in \mathcal{M}(\mathbb{R}^{d})$. Then 
\begin{equation}
\lim_{Q\rightarrow \infty }\strokedint_{Q}\widehat{\nu }(\eta )d\eta =c\nu
(\{0\}),  \label{Wien-0}
\end{equation}%
for some non-zero constant $c$. Also, for any two functions $\xi :(0,\infty
)\rightarrow \mathbb{R}^{d}$ and $r:(0,\infty )\rightarrow (0,\infty )$ with 
$\lim_{t\rightarrow \infty }r(t)=\infty $ we have 
\begin{equation}
\lim_{t\rightarrow \infty }\strokedint_{B(\xi (t),r(t))}\widehat{\nu }(\eta
)d\eta =c\nu (\{0\}).  \label{Wien-00}
\end{equation}
\end{lemma}

\noindent \textbf{Proof.} (The proof of the first part is also explicitly
written in \cite[Proposition 3.1]{P}; we present it here for the convenience
of the reader.) For $Q=[q_{1},r_{1})\times ...\times \lbrack q_{d},r_{d})$
(where $q_{1}<r_{1}$, ..., $q_{d}<r_{d}$ are some real numbers) we have 
\begin{equation*}
|Q|^{-1}(\mathbf{1}_{Q})^{\vee }(x)=c_{1}\prod_{j:x_{j}\neq 0}\frac{%
e^{ir_{j}x_{j}}-e^{iq_{j}x_{j}}}{ix_{j}(r_{j}-q_{j})},
\end{equation*}%
for any $x\in \mathbb{R}^{d}$, with the convention that the product over an
empty set equals $c_{1}$. Hence, the continuous function $|Q|^{-1}(\mathbf{1}%
_{Q})^{\vee }$ is bounded by $c_{1}$, equals $c_{1}$ at the origin and, when 
$Q\rightarrow \infty $, we have\ $|Q|^{-1}(\mathbf{1}_{Q})^{\vee
}\rightarrow 0$ uniformly, out of any neighborhood of the origin.

As in the proof of \cite[Theorem 7.13]{K} we get that, 
\begin{equation}
\lim_{Q\rightarrow \infty }\left\langle |Q|^{-1}(\mathbf{1}_{Q})^{\vee },\nu
-\nu (\{0\})\delta _{0}\right\rangle =0.  \label{miuA-lim}
\end{equation}

We also have 
\begin{eqnarray*}
\left\langle |Q|^{-1}(\mathbf{1}_{Q})^{\vee },\nu -\nu (\{0\})\delta
_{0}\right\rangle &=&\left\langle |Q|^{-1}(\mathbf{1}_{Q})^{\vee },\nu
\right\rangle -\left\langle |Q|^{-1}(\mathbf{1}_{Q})^{\vee },\nu
(\{0\})\delta _{0}\right\rangle \\
&=&\left\langle |Q|^{-1}\mathbf{1}_{Q},\widehat{\nu }\right\rangle -c%
\overline{\nu (\{0\})},
\end{eqnarray*}%
which together with (\ref{miuA-lim}) gives 
\begin{equation*}
c\overline{\nu (\{0\})}=\lim_{Q\rightarrow \infty }\left\langle |Q|^{-1}%
\mathbf{1}_{Q},\widehat{\nu }\right\rangle ,
\end{equation*}%
which is equivalent to (\ref{Wien-0}).

The equality (\ref{Wien-00}) can be directly obtained from (\ref{Wien-0}) by
using Withney decompositions of the balls $B(\xi (t),r(t))$ as follows.
Given $\varepsilon >0$ there exists some positive integer $N_{\varepsilon }$
(independent of $t$) and almost disjoint cubes $Q_{t}^{1},..,Q_{t}^{N_{%
\varepsilon }}$ such that $Q_{t}^{1},..,Q_{t}^{N_{\varepsilon }}\subset
B(\xi (t),r(t))$, each quotient $q_{j}:=|Q_{t}^{j}|/|B(\xi (t),r(t))|$
constant in $t$ and 
\begin{equation*}
1-\varepsilon <\sum_{j=1}^{N_{\varepsilon }}q_{j}<1,
\end{equation*}%
for any $t>0$. Hence, by applying (\ref{Wien-0}) on each cube $Q_{t}^{j}$,
we have 
\begin{eqnarray*}
\strokedint_{B(\xi (t),r(t))}\widehat{\nu }(\eta )d\eta
&=&\sum_{j=1}^{N_{\varepsilon }}q_{j}\strokedint_{Q_{t}^{j}}\widehat{\nu }%
(\eta )d\eta +\varepsilon O_{t\rightarrow \infty }(1) \\
&=&c\nu (\{0\})\sum_{j=1}^{N_{\varepsilon }}q_{j}+\varepsilon
O_{t\rightarrow \infty }(1)+o_{t\rightarrow \infty }(1) \\
&=&c\nu (\{0\})+\varepsilon O_{t\rightarrow \infty }(1)+o_{t\rightarrow
\infty }(1),
\end{eqnarray*}%
and this proves (\ref{Wien-00}).\hfill $\square $

\bigskip

Using Lemma \ref{lem.Wien} and the above computations involving $\mu \ast
\mu ^{\sharp }$ we obtain Wiener's theorem

\begin{theorem}
\label{th.W}Let $\mu \in \mathcal{M}(\mathbb{R}^{d})$. Then 
\begin{equation}
\lim_{Q\rightarrow \infty }\strokedint_{Q}\left\vert \widehat{\mu }(\eta
)\right\vert ^{2}d\eta =c\sum_{\tau \in \mathbb{R}^{d}}\left\vert \mu
(\{\tau \})\right\vert ^{2}.  \label{miuA-W}
\end{equation}

Also, for any two functions $\xi :(0,\infty )\rightarrow \mathbb{R}^{d}$ and 
$r:(0,\infty )\rightarrow (0,\infty )$ with $\lim_{t\rightarrow \infty
}r(t)=\infty $ we have 
\begin{equation}
\lim_{t\rightarrow \infty }\strokedint_{B(\xi (t),r(t))}\left\vert \widehat{%
\mu }(\eta )\right\vert ^{2}d\eta =c\sum_{\tau \in \mathbb{R}^{d}}\left\vert
\mu (\{\tau \})\right\vert ^{2}.  \label{miuaA-W-1}
\end{equation}
\end{theorem}

\textbf{Proof.} Let $\nu $ be the measure $\nu =\mu \ast \mu ^{\sharp }$.
Using (\ref{Wien-0}), by taking into account that $\widehat{\nu }=\left\vert 
\widehat{\mu }\right\vert ^{2}$ and (see (\ref{miuA-conv})) 
\begin{equation*}
\overline{\nu (\{0\})}=c\sum_{\tau \in \mathbb{R}^{d}}\left\vert \mu (\{\tau
\})\right\vert ^{2},
\end{equation*}%
we obtain (\ref{miuA-W}). Similarly, (\ref{miuaA-W-1}) follows from (\ref%
{Wien-0}).\hfill $\square $

\begin{remark}
Versions of this result appear in the work of Wiener (see for instance \cite[%
Theorem 24, p. 1461]{W}), however, the sets on which averages are considered
are with fixed centers. What is important for the applications in this paper
is the fact that in Lemma \ref{lem.Wien} and Theorem \ref{th.W} (for which
the proof is in essence the same as Wiener's proof) we allow the centers of
the cubes or of the balls to vary in the limiting process. It seems that the
first mention of this aspect appears in the older version of Katznelson's
book \cite{K} (from 1963). Surprisingly, the first application that takes
into account the moving centers seems to be given in the work of Pe\l %
{}czy\' nski in \cite[p. 401 and Theorem 3.1]{P} (1989) who credits the idea
to J.-P. Kahane .
\end{remark}

\subsection{The Bonami-Mohanty transference result revisited}

Here we present a slight generalisation of Theorem 3.4 in \cite{Bon-Moh} to
a wider class of Banach function spaces.

Let $X(\mathbb{R}^{d})$ be a Banach space of tempered distributions on $%
\mathbb{R}^{d}$ with the following properties:

\bigskip

\begin{itemize}
\item[(i)] $X(\mathbb{R}^{d})$ is translation invariant, i.e., 
\begin{equation*}
\left\Vert f(\cdot +v)\right\Vert _{X(\mathbb{R}^{d})}\lesssim \left\Vert
f\right\Vert _{X(\mathbb{R}^{d})},
\end{equation*}%
for any $f\in X(\mathbb{R}^{d})$ and any $v\in \mathbb{R}^{d}$ (where the
implicit constant does not depend on $f$ or $v$).

\bigskip

\item[(ii)] If $\varphi \in C_{c}^{\infty }(\mathbb{R}^{d})$, then $\varphi
f\in X(\mathbb{R}^{d})$ and 
\begin{equation*}
\left\Vert \varphi f\right\Vert _{X(\mathbb{R}^{d})}\lesssim _{\varphi
}\left\Vert f\right\Vert _{X(\mathbb{R}^{d})},
\end{equation*}%
for any $f\in X(\mathbb{R}^{d})$.

\bigskip

\item[(iii)] Suppose $(f_{n})_{n\geq 1}$, $(g_{n})_{n\geq 1}$ are two
sequences in $X(\mathbb{R}^{d})$, with 
\begin{equation*}
(\supp f_{n_{1}})\cap (\supp f_{n_{2}})=(\supp
g_{n_{1}})\cap (\supp g_{n_{2}})=\varnothing ,
\end{equation*}
for any $n_{1}\neq n_{2}$ and $\left\Vert f_{n}\right\Vert _{X(\mathbb{R}%
^{d})}\leq \left\Vert g_{n}\right\Vert _{X(\mathbb{R}^{d})}$, for any $n\geq
1$. Then,%
\begin{equation*}
\left\Vert \sum_{n=1}^{\infty }f_{n}\right\Vert _{X(\mathbb{R}^{d})}\leq
\left\Vert \sum_{n=1}^{\infty }g_{n}\right\Vert _{X(\mathbb{R}^{d})}.
\end{equation*}

\item[(iv)] There exists some positive integer $N$ such that, if $e_{\chi
}:=\exp (2\pi i\left\langle \chi ,\cdot \right\rangle) $, for $\chi \in 
\mathbb{Z}^{d} $, then 
\begin{equation*}
\left\Vert e_{\chi }f\right\Vert _{X(\mathbb{T}^{d})}\lesssim \left\langle
\chi \right\rangle ^{N}\left\Vert f\right\Vert _{X(\mathbb{T}^{d})},
\end{equation*}%
for any $f\in X(\mathbb{T}^{d})$. Here, $X(\mathbb{T}^{d})$ is the normed
space of distributions on $\mathbb{T}^{d}$ given by the norm%
\begin{equation*}
\left\Vert f\right\Vert _{X(\mathbb{T}^{d})}:=\left\Vert \varphi _{0}%
\widetilde{f}\right\Vert _{X(\mathbb{R}^{d})},
\end{equation*}%
where $\widetilde{f}$ is the periodic extension of $f$ to $\mathbb{R}^{d}$
and $\varphi _{0}\in C_{c}^{\infty }([-8,8]^{d})$ is a fixed function such
that $\varphi _{0}\equiv 1$ on $[-4,4]^{d}$ (we identify $\mathbb{T}^{d}$
with $Q_{0}:=[-1,1]^{d}$).
\end{itemize}

\bigskip

\begin{remark}
Examples of spaces satisfying the above conditions are the classical Sobolev
space $W^{s,p}(\mathbb{R}^{d})$, for $p\in [1,\infty]$, as well as $%
C_{0}^{l}(\mathbb{R}^{d})$.
\end{remark}

\begin{proposition}
\label{Bon-Moh-general}Suppose $X(\mathbb{R}^{d})$ is a Banach function
space with the properties (i)--(iv) above. Let $K$ be a distribution on the
torus such that $T_{K}\in M(X(\mathbb{T}^{d}))$ and let $\psi $ be a
Schwartz function on $\mathbb{R}^{d}$. Then, $T_{\widetilde{K}\psi }\in M(X(%
\mathbb{R}^{d}))$.
\end{proposition}

It of interest to deal with extension of the form $\tilde{K}\psi $ where $%
\psi $ is any Schwartz function, not necessarily compactly supported as in 
\cite{Bon-Moh}. For this purpose we will use the following elementary lemma.

\begin{lemma}
\label{lem.sch}Let $\psi $ be a Schwartz function on $\mathbb{R}^{d}$. There
exist smooth functions $\psi _{1},\psi _{2},...$ on $\mathbb{R}^{d}$, each
one of them supported in a translation of the cube $(1/4)Q_{0}$, such that $%
\psi _{1}+\psi _{2}+...=\psi $ and 
\begin{equation}
\sum_{n=1}^{\infty }\sum_{\chi \in \mathbb{Z}^{d}}\left\langle \chi
\right\rangle ^{2N}|\widehat{\psi }_{n}(\chi )|<\infty ,  \label{sch-0}
\end{equation}%
where $\widehat{\psi }_{n}$ is the Fourier transform of $\psi _{n}$
considered on $\mathbb{R}^{d}$.
\end{lemma}

\noindent\textbf{Proof.} Consider first an arbitrary function $\varphi \in
C_{c}^{\infty }((1/4)Q_{0})$, and for any $y\in \mathbb{R}^{d}$ denote by $%
\varphi _{y}$ the function $\varphi _{y}=\varphi (\cdot +y)$. We have 
\begin{equation*}
\widehat{\psi \varphi _{y}}(\xi )=\widehat{\psi }\ast \widehat{\varphi _{y}}%
(\xi )=\int_{\mathbb{R}^{d}}e^{i\left\langle y,\eta \right\rangle }\widehat{%
\varphi }(\eta )\widehat{\psi }(\xi -\eta )d\eta =\widehat{F}_{\xi }(-y),
\end{equation*}%
where $F_{\xi }:=\widehat{\varphi }(\cdot )\widehat{\psi }(\xi -\cdot )$ is
a Schwartz function. We have 
\begin{equation*}
|\widehat{F}_{\xi }(y)|\lesssim \left\langle y\right\rangle
^{-(d+1)}(\left\Vert F_{\xi }\right\Vert _{L^{1}}+\left\Vert \nabla
^{m}F_{\xi }\right\Vert _{L^{1}})\lesssim \left\langle y\right\rangle
^{-(d+1)}\sum_{j_{1},j_{2}\leq d+1}\left\Vert |\nabla ^{j_{1}}\widehat{%
\varphi }||\nabla ^{j_{2}}\widehat{\psi }(\xi -\cdot )|\right\Vert _{L^{1}}.
\end{equation*}

Hence, 
\begin{eqnarray*}
\sum_{\chi ^{\prime }\in \mathbb{Z}^{d}}\sum_{\chi \in \mathbb{Z}%
^{d}}\left\langle \chi \right\rangle ^{2N}|\widehat{\psi \varphi _{\chi
^{\prime }}}(\chi )| &=&\sum_{\chi ^{\prime }\in \mathbb{Z}^{d}}\sum_{\chi
\in \mathbb{Z}^{d}}\left\langle \chi \right\rangle ^{2N}|\widehat{F}_{\chi
}(-\chi ^{\prime })| \\
&\lesssim &\sum_{j_{1},j_{2}\leq d+1}\sum_{\chi ^{\prime }\in \mathbb{Z}%
^{d}}\left\langle \chi ^{\prime }\right\rangle ^{-(d+1)}\sum_{\chi \in 
\mathbb{Z}^{d}}\left\langle \chi \right\rangle ^{2N}\left\Vert |\nabla
^{j_{1}}\widehat{\varphi }||\nabla ^{j_{2}}\widehat{\psi }(\chi -\cdot
)|\right\Vert _{L^{1}},
\end{eqnarray*}%
and since 
\begin{equation*}
\sum_{\chi \in \mathbb{Z}^{d}}\left\langle \chi \right\rangle
^{2N}\left\Vert |\nabla ^{j_{1}}\widehat{\varphi }||\nabla ^{j_{2}}\widehat{%
\psi }(\chi -\cdot )|\right\Vert _{L^{1}}\lesssim \int_{\mathbb{R}%
^{d}}\left( \sum_{\chi \in \mathbb{Z}^{d}}\frac{\left\langle \chi
\right\rangle ^{2N}}{\left\langle \chi -x\right\rangle ^{2N+d+1}}\right)
|\nabla ^{j_{1}}\widehat{\varphi }|(x)dx\lesssim 1,
\end{equation*}%
we get 
\begin{equation}
\sum_{\chi ^{\prime }\in \mathbb{Z}^{d}}\sum_{\chi \in \mathbb{Z}%
^{d}}\left\langle \chi \right\rangle ^{2N}|\widehat{\psi \varphi _{\chi
^{\prime }}}(\chi )|<\infty .  \label{sch-1}
\end{equation}

Now one can deduce (\ref{sch-0}) from (\ref{sch-1}) by considering a smooth
partition of unity on $\mathbb{R}^{d}$ with a family of smooth functions of
the form $(\varphi _{\chi }^{j})_{j\in \{1,...,s\},\chi \in \mathbb{Z}^{d}}$%
, where each $\varphi ^{j}$ is supported in $(1/4)Q_{0}$. \ One can obtain
the family of functions $(\psi _{n})_{n\geq 1}$\ by relabeling the functions 
$\psi \varphi _{\chi ^{\prime }}$,\ $\chi \in \mathbb{Z}^{d}$.\hfill $%
\square $

\noindent \textbf{Proof of Proposition \ref{Bon-Moh-general}.} The proof
closely follows the proof of Theorem 3.4 in \cite{Bon-Moh}. First, we
observe that, by property (iv) we have 
\begin{equation}
\left\Vert \widehat{K}(D-\chi )f\right\Vert _{X(\mathbb{T}^{d})}=\left\Vert
e_{\chi }\widehat{K}(D)(e_{-\chi }f)\right\Vert _{X(\mathbb{T}^{d})}\lesssim
_{K}\left\langle \chi \right\rangle ^{2N}\left\Vert f\right\Vert _{X(\mathbb{%
T}^{d})},  \label{BM-1}
\end{equation}%
for any $f\in X(\mathbb{T}^{d})$ and any $\chi \in \mathbb{Z}^{d}$. By the
identity%
\begin{equation*}
\widehat{K\varphi }=\sum_{\chi \in \mathbb{Z}^{d}}\widehat{K}(\cdot -\chi )%
\widehat{\varphi }(\chi ),
\end{equation*}%
an the triangle inequality, (\ref{BM-1}) implies that, 
\begin{equation}
\left\Vert \widehat{K\varphi }(D)f\right\Vert _{X(\mathbb{T}^{d})}\lesssim
\left( \sum_{\chi \in \mathbb{Z}^{d}}\left\langle \chi \right\rangle ^{2N}|%
\widehat{\varphi }(\chi )|\right) \left\Vert f\right\Vert _{X(\mathbb{T}%
^{d})},  \label{BM-2}
\end{equation}%
for any $f\in X(\mathbb{T}^{d})$ and any $\varphi \in C^{\infty }(\mathbb{T}%
^{d})$ (here we used the Fourier transform on $\mathbb{T}^{d}$).

Consider some smooth functions $\theta _{1},..,\theta _{s}$ on $\mathbb{T}%
^{d}$ ($s$ is an integer depending on $d$), such that $\theta
_{1}+...+\theta _{s}=1$, on $\mathbb{T}^{d}$, and each $\theta _{j}$ is
supported in a cube of side length $<1/4$. If $\widetilde{\theta }^{1},..,%
\widetilde{\theta }^{s}$ are the periodic extensions of $\theta
^{1},..,\theta ^{s}$ respectively, then $\widetilde{\theta }_{1}+..+%
\widetilde{\theta }_{s}=1$, on $\mathbb{R}^{d}$. We also consider, for each $%
\chi \in \mathbb{Z}^{d}$, the function $\widetilde{\theta }_{j}^{\chi }\in
C_{c}^{\infty }(\mathbb{R}^{d})$ that is equal to $\widetilde{\theta }_{j}$
on $Q_{0}+\chi $ and $0$ outside $Q_{0}+\chi $. Now, for a given $f\in X(%
\mathbb{R}^{d})$ we can write the decomposition%
\begin{equation}
f=\sum_{j=1}^{s}f_{j}\text{, \ \ and \ }f_{j}=\sum_{\chi \in \mathbb{Z}%
^{d}}f_{j}^{\chi },  \label{BM-3}
\end{equation}%
where $f_{j}:=\widetilde{\theta }_{j}f$, and $f_{j}^{\chi }:=\widetilde{%
\theta }_{j}^{\chi }f$ .

By Lemma \ref{lem.sch} we can choose some smooth functions $\psi _{1},\psi
_{2},...$ on $\mathbb{R}^{d}$, each one of them supported in a translation
of the cube $(1/4)Q_{0}$, such that $\psi _{1}+\psi _{2}+...=\psi $ and 
\begin{equation}
\sum_{n=1}^{\infty }\sum_{\chi \in \mathbb{Z}^{d}}\left\langle \chi
\right\rangle ^{2N}|\widehat{\psi }_{n}(\chi )|<\infty ,  \label{BM-3'}
\end{equation}%
where $\widehat{\psi }_{n}$ is the Fourier transform of $\psi _{n}$
considered on $\mathbb{R}^{d}$.

Fix some integer $n\geq 1$ and some $j\in \{1,...,s\}$. Since the support of
each $(\widetilde{K}\psi _{n})\ast f_{j}^{\chi }$ is included on a cube of
the form $(1/2)Q_{0}+z$, for some $z\in \mathbb{R}^{d}$ (depending on $n$
and $\chi $), we have 
\begin{eqnarray*}
(\widetilde{K}\psi _{n})\ast f_{j}^{\chi }(x) &=&\int_{Q_{0}+z}\widetilde{K}%
(y)\psi _{n}(y)f_{j}^{\chi }(x-y)dy \\
&=&\int_{\mathbb{T}^{d}}K(y)\psi _{n}(y+z)f_{j}^{\chi }(x-y-z)dy,
\end{eqnarray*}%
and by (\ref{BM-2}) together with (i), (iv) we get%
\begin{equation*}
\left\Vert (\widetilde{K}\psi _{n})\ast f_{j}^{\chi }\right\Vert _{X(\mathbb{%
R}^{d})}\lesssim \left( \sum_{\chi \in \mathbb{Z}^{d}}\left\langle \chi
\right\rangle ^{2N}|\widehat{\psi }_{n}(\chi )|\right) \left\Vert
f_{j}^{\chi }\right\Vert _{X(\mathbb{R}^{d})}.
\end{equation*}

Now, using (\ref{BM-3}), the fact that the distributions $(\widetilde{K}\phi
_{n})\ast f_{j}^{\chi }$ have pair-wise disjoint supports and (iii) we have%
\begin{equation*}
\left\Vert (\widetilde{K}\psi _{n})\ast f_{j}\right\Vert _{X(\mathbb{R}%
^{d})}\lesssim \left( \sum_{\chi \in \mathbb{Z}^{d}}\left\langle \chi
\right\rangle ^{2N}|\widehat{\psi }_{n}(\chi )|\right) \left\Vert
f_{j}\right\Vert _{X(\mathbb{R}^{d})}.
\end{equation*}

This gives us 
\begin{equation*}
\left\Vert (\widetilde{K}\psi )\ast f_{j}\right\Vert _{X(\mathbb{R}%
^{d})}\lesssim \left( \sum_{n=1}^{\infty }\sum_{\chi \in \mathbb{Z}%
^{d}}\left\langle \chi \right\rangle ^{2N}|\widehat{\psi }_{n}(\chi
)|\right) \left\Vert f_{j}\right\Vert _{X(\mathbb{R}^{d})},  \label{BM-4}
\end{equation*}
which together with (\ref{BM-3'}), (\ref{BM-3}) and (ii) concludes the
proof. \hfill $\square $

\end{document}